\newif\ifarxiv
\arxivtrue  % arXiv layout
\ifarxiv
\documentclass[12pt]{article}
\usepackage{amsmath}
\usepackage{amssymb}
\usepackage{amsthm}
\usepackage{amsfonts}
\usepackage{amscd}
\usepackage[dvipsnames]{xcolor}
\usepackage[T1]{fontenc} % removes warning: some font shapes were not available
\usepackage{graphicx}
\usepackage{enumitem} % http://tex.stackexchange.com/questions/86054/how-to-remove-the-whitespace-before-itemize-enumerate [noitemsep,topsep=0pt]
\usepackage{caption}
\usepackage[numbers,sort&compress]{natbib}
\usepackage{hyperref}
\usepackage{tikz}
\usepackage{geometry}
\usepackage[medium]{titlesec}
\usepackage{etoolbox} % section number patch
\usepackage{arydshln} % hdashline
\usepackage[capitalise]{cleveref} % \Cref
\usepackage[nottoc]{tocbibind}  % include bibliography in toc
\usepackage{tocloft} % table of contents in text
\newtheorem{theorem}{Theorem}
\newtheorem{corollary}{Corollary}

\newtheorem{lemma}{Lemma}
\newtheorem{proposition}[lemma]{Proposition}
\newtheorem{conjecture}{Conjecture}
\theoremstyle{definition}
\newtheorem{definition}[lemma]{Definition}
\newtheorem{example}[lemma]{Example}
\newtheorem{remark}[lemma]{Remark}
\newtheorem{notation}[lemma]{Notation}
\hypersetup{colorlinks,urlcolor=black!45!green,citecolor=black!45!green,linkcolor=RoyalBlue}
\makeatletter\patchcmd{\ttlh@hang}{\parindent\z@}{\parindent\z@\leavevmode}{}{}\patchcmd{\ttlh@hang}{\noindent}{}{}{}\makeatother % section number patch
\titlespacing*{\section}{0pt}{0mm}{0mm}
\titlespacing*{\subsection}{0pt}{0mm}{0mm}
\newcommand{\myspace}{\setlength{\abovedisplayskip}{1mm}\setlength{\belowdisplayskip}{0mm}}
\newenvironment{Mlist}{\begin{itemize}[topsep=0pt,itemsep=0pt,leftmargin=7mm]}{\end{itemize}}
\newenvironment{Menum}[2]{\begin{enumerate}[topsep=0pt,itemsep=0pt,leftmargin=9mm,label=#1,ref=#2]}{\end{enumerate}}
\newenvironment{claims}{\begin{Menum}{\textnormal{\textbf{(\alph*)}}}{(\alph*)}}{\end{Menum}}
\else
\documentclass[oneeqnum,onefignum,onetabnum,onethmnum]{siamart220329}
\usepackage{amssymb}
\usepackage{enumitem} % http://tex.stackexchange.com/questions/86054/how-to-remove-the-whitespace-before-itemize-enumerate [noitemsep,topsep=0pt]
\usepackage[numbers,sort&compress]{natbib}
\usepackage{tikz}
\usepackage{arydshln} % hdashline
\newsiamremark{remark}{Remark}
\newsiamremark{example}{Example}
\newsiamremark{conjecture}{Conjecture}
\newsiamremark{notation}{Notation}
\newenvironment{Mlist}{\begin{itemize}}{\end{itemize}}
\newenvironment{Menum}[2]{\begin{enumerate}[label=#1,ref=#2]}{\end{enumerate}}
\newenvironment{claims}{\begin{Menum}{\textnormal{(\alph*)}}{(\alph*)}}{\end{Menum}}
\title{Shapes of surfaces that contain a great and a small circle through each point
\thanks{\funding{This work was funded by the Austrian Science Fund (FWF) projects P33003 and P36689.}}}
\author{Niels Lubbes
\thanks{Institute for Algebra, Johannes Kepler University, Linz, Austria
\email{info@nielslubbes.com}}}
\fi

\usetikzlibrary{calc}
\def\centerarc[#1](#2)(#3:#4:#5:#6){\draw[#1] ($(#2)+({#5*cos(#3)},{#5*sin(#3)})$) arc (#3:#4:#5 and #6);}% Syntax: [draw options] (center) (initial angle:final angle:radius)
\definecolor{colG}{RGB}{8,144,8}
\newcommand{\fig}[3]{\includegraphics[height=#1cm, width=#2cm]{img/#3}}
\newcommand{\csep}[1]{\setlength{\tabcolsep}{#1}}
\newcommand{\ASN}[1]{Assertion~\ref{#1}}
\newcommand{\AXM}[1]{Axiom~\ref{#1}}
\newcommand{\END}{\hfill $\vartriangleleft$}
\newcommand{\df}[1]{{\it #1}}
\newcommand{\Wlog}{without loss of generality }
\newcommand{\resp}{respectively}
\newcommand{\st}{such that }
\newcommand{\wrt}{with respect to }

\renewcommand{\c}{\colon}
\newcommand{\set}[2]{\{#1:#2\}}
\newcommand{\dto}{\dasharrow}

\newcommand{\bas}[1]{\langle #1\rangle}
\newcommand{\aut}{\operatorname{Aut}}
\newcommand{\Sing}{\operatorname{Sing}}
\renewcommand{\P}{{\mathbb{P}}}
\renewcommand{\S}{{\mathbb{S}}}
\newcommand{\C}{{\mathbb{C}}}
\newcommand{\R}{{\mathbb{R}}}
\newcommand{\Z}{{\mathbb{Z}}}

\newcommand{\E}{{\mathbb{E}}}

\newcommand{\cH}{{\mathcal{H}}}

\newcommand{\cO}{{\mathcal{O}}}
\newcommand{\cZ}{{\mathcal{Z}}}
\newcommand{\cS}{{\mathcal{S}}}
\newcommand{\X}{{\mathcal{X}}}

\renewcommand{\k}{k}
\newcommand{\h}{h}
\renewcommand{\l}{\ell}

\newcommand{\id}{\operatorname{id}}

\newcommand{\oL}{\overline{L}}
\newcommand{\oR}{\overline{R}}
\newcommand{\bR}{\mathbf{R}}
\newcommand{\hstar}{\hat{\star}}
\newcommand{\lt}{\operatorname{LT}}
\newcommand{\rt}{\operatorname{RT}}
\newcommand{\cL}{{\mathcal{L}}}
\newcommand{\cR}{{\mathcal{R}}}

\begin{document}

\ifarxiv
%%%%% START ARXIV %%%%%
\myspace
\begin{center}
{\LARGE
Shapes of surfaces that contain a great and a small circle through each point
}
\\[3mm]
{\large
Niels Lubbes
}
\\[3mm]
{\today}
\end{center}

\begin{abstract}
We classify the topological types of surfaces in the 3-dimensional unit sphere that contain both a great and a small circle through each point. In particular, these surfaces are homeomorphic to one of five normal forms and are either the pointwise product of circles in the unit quaternions or contain five concurrent circles. We classify the real singular loci of such surfaces and characterize how circles in the surface meet the self-intersection locus.
\\[2mm]
\\[2mm]
{\bf MSC2010:}
51B10, % geometry -- nonlinear incidence geometry -- Moebius geometries
51M15, % geometry -- real and complex geometry -- geometric constructions
14J17, % algebraic geometry -- surfaces and higher dim. varieties -- singularities
14P25 % algebraic geometry -- real algebraic and analytic geometry -- topology of real algebraic varieties
% 14C20  % algebraic geometry -- cycles and subschemes -- divisors, linear systems, invertible sheafs
% 14C20, % algebraic geometry -- cycles and subschemes -- divisors, linear systems, invertible sheafs
% 14C21, % algebraic geometry -- cycles and subschemes -- pencils, nets, webs
% 14C22, % algebraic geometry -- cycles and subschemes -- Picard groups
% 14J17  % algebraic geometry -- surfaces and higher dim. varieties -- singularities
% 14J26, % algebraic geometry -- surfaces and higher dim. varieties -- rational and ruled surfaces
% 14N25, % algebraic geometry -- projective and enumrative geometry -- varieties of low degree
% 14P25, % algebraic geometry -- real algebraic and analytic geometry -- topology of real algebraic varieties
% 14Q10, % algebraic geometry -- computational aspects in algebraic geometry -- surfaces, hypersurfaces
% 51B10, % geometry -- nonlinear incidence geometry -- Moebius geometries
% 51M15, % geometry -- real and complex geometry -- geometric constructions
% 53A05, % differential geometry -- classical differential geometry -- Surfaces in Euclidean space
% 53A20, % differential geometry -- classical differential geometry -- projective differential geometry
% 53A30, % differential geometry -- classical differential geometry -- conformal differential geometry
% 53A35, % differential geometry -- classical differential geometry -- Non-Euclidean differential geometry
% 53A60, % differential geometry -- classical differential geometry -- geometry of webs
\end{abstract}
%%%%% END ARXIV %%%%%
\else
%%%%% START SIAGA %%%%%
\maketitle
\begin{abstract}
We classify the topological types of surfaces in the 3-dimensional unit sphere that contain both a great and a small circle through each point. In particular, these surfaces are homeomorphic to one of five normal forms and are either the pointwise product of circles in the unit quaternions or contain five concurrent circles. We classify the real singular loci of such surfaces and characterize how circles in the surface meet the self-intersection locus.
\end{abstract}
%%%%% END SIAGA %%%%%
\fi

\section{Introduction}
\label{sec:intro}

Can we recover the shape of an embedded surface from the knowledge
that this surface contains two curves of predefined type through each point?

For example, a surface in $\R^3$ containing $\lambda\geq 3$ lines through each point
must be a plane.
A surface containing $\lambda=2$ lines through each point
must be a doubly ruled quadric and thus shaped like a horse saddle or a cooling tower.
Such hyperboloid structures are of interest to architects \cite{2007arch}.
A surface in~$\R^3$ containing a line and circle through each
point is also of degree at most two~\cite{2013} and therefore either a plane, cone, cylinder or one-sheeted hyperboloid.

Now let us consider the analogue of lines and circles in elliptic geometry,
namely great and small circles in the \df{unit sphere}~$S^3\subset\R^4$.
A circle in~$S^3$ is called \df{great} if centered at the origin and \df{small} otherwise.

If a surface in $S^3$ contains
$\lambda\geq 3$ great circles through each point, then it must be a sphere,
and if this surface contains $\lambda=2$ great circles through each point,
then it is homeomorphic to a torus by \citep[Corollary~3]{2024lt}.
See \cite{1985}, in case the great circles belong to some Hopf fibration.

In this article, we consider the shape of surfaces that
contain both a great circle and a small circle through each point.
In \Cref{fig:shape}, we see possible shapes of
stereographic projections of such surfaces.
\begin{figure}[!ht]
\centering
\csep{4mm}
\begin{tabular}{ccc}
\fig{3}{3.5}{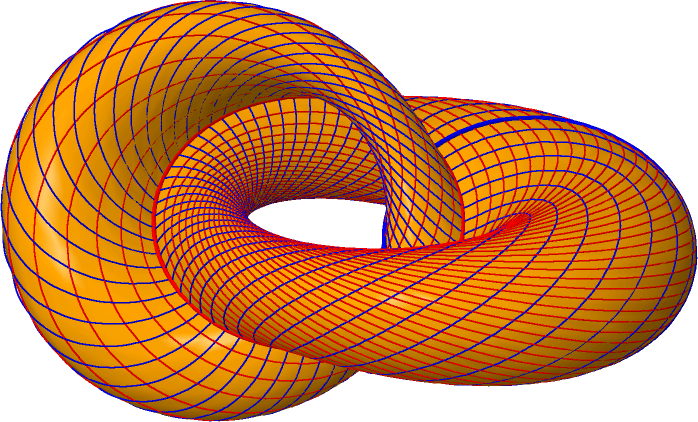}  &
\fig{3}{3.5}{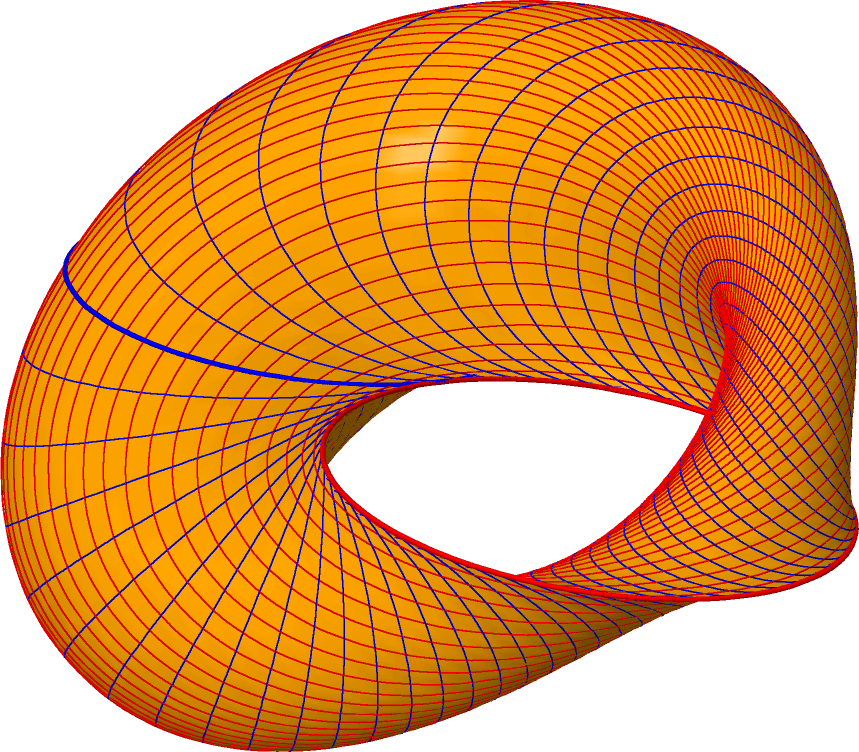}      &
\fig{3}{3.5}{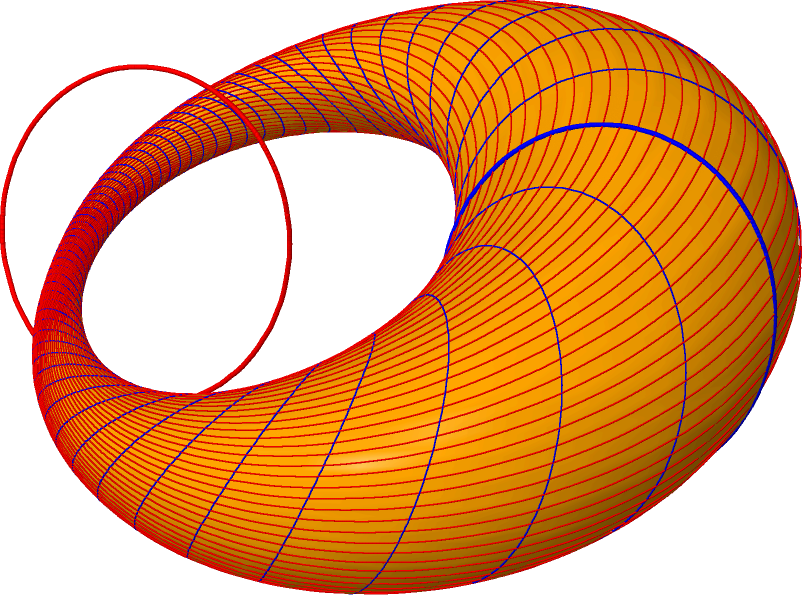}
\\
Shape~\ref{I} & Shape~\ref{II} & Shape~\ref{III}
\end{tabular}
\caption{Stereographic projections of surfaces in $S^3$
that contain a great (red) and a small (blue) circle through each point.}
\label{fig:shape}
\end{figure}

The classification of smooth abstract rational surfaces up to homeomorphism
has been completed by Comessatti \citep[1914]{1914}\citep[3.6.1]{2000k}.
The classification of topological types of algebraic surfaces in real projective 3-space
is open and part of \df{Hilbert's 16th problem}. For smooth surfaces of degree at most
four this classification was completed by Kharlamov using homological methods
(see \citep[3.5.4]{2000k} for an overview). In this article, we do not assume smoothness
and consider surfaces of degrees up to eight.

A surface $Z\subset S^3$ is \df{$\lambda$-circled} if $Z$
contains at least $\lambda\in \Z_{\geq 0}\cup\{\infty\}$ circles through a general point.
If $\lambda\in\Z_{\geq0}$, then we assume that $Z$ is not $(\lambda+1)$-circled.
We call $Z$ \df{celestial} if it is $\lambda$-circled \st $\lambda\geq2$.
We call $Z$ \df{great} if it contains a great circle through a general point.

To explore potential applications of our results, let us consider the following problem
from geometric modeling which orginates from
the automotive, airplane and ship industry:
\textit{Given a set of algebraic surfaces $\cS$, classify smooth $C^1$-surfaces~$X$ \st
$X\subset S_1\cup\cdots\cup S_n$ for some $S_1,\ldots,S_n\in \cS$.}

There is an extensive literature for the case that
$\cS$ consists of celestial surfaces of degree four whose curvature lines are circles.
See for example \cite{1988,1990,1995,2002,2011kb,2012dup}
where $X\cap S_i$ is called a ``cyclidic patch'' of the ``Dupin cyclide''~$S_i$
and the great circles in $S_i$ are called ``Villarceau circles''.
In this case, the topological surface~$X$ may contain up to four circular arcs through almost each point.
This connects to architecture since $X$ admits embedded graphs whose edges are circular arcs
and such arcs are suitable for structural beams and panel boundaries of buildings~\cite{2011}.

With such architectural applications in mind, \cite{2012web} allows $\cS$ to contain all quartic celestial surfaces
so that $X$ can have up to six circular arcs through each point.
The topological classification in \cite{2019} of these surfaces
sheds additional light on the possible shapes of~$X$.

In this article, we propose for $\cS$ to be the set of all great celestial surfaces.
We classify such surfaces up to homeomorphisms
and show that their singular loci are contained in some great circle.
Their classification up to diffeomorphisms of $S^3$ remains an open problem,
but we state a precise conjecture.
Up to M\"obius equivalence, such $\cS$ is a superset of the smooth quartic celestial surfaces \cite{2024lt}
and thus $X$ admits a wide range of shapes.
However, if $\deg S_i>4$,
then $X\cap S_i$ contains only two circular arcs through each point~\cite{2021circle}.

We remark that aside architecture, celestial surfaces have also potential applications
in kinematics~\cite{2018kin,2020kin} and computer vision~\cite{2023vision,2024lc}.

Before we state our main result \Cref{thm:shape}, let us first introduce some terminology and state
its three \Cref{cor:top,cor:prod,cor:AB}.

From now on surfaces can be defined as zero sets of polynomials with real coefficients
and therefore with \df{surface} is meant a real irreducible algebraic surface.

Unless explicitly indicated otherwise, a \df{cycle}, \df{sphere}, \df{torus}, \df{disk}
and \df{solid torus}
is a topological space that is homeomorphic to $S^1$, $S^2$, $S^1\times S^1$,
$D^2:=\set{z\in\C}{|z|\leq 1}$ and $S^1\times D^2$, \resp.
Notice that a \df{circle} is a non-empty irreducible conic in~$S^3$ and thus not invariant under homeomorphisms.

A cycle $C\subset M$ in a topological space~$M$ is called \df{trivial} in~$M$ if it bounds
an embedded disk~$D\subset M$.
We remark that any two non-trivial cycles in a torus~$T$ are related by some homeomorphism~$T\to T$
(see \citep[Theorem~C13]{2003}).

By \citep[Proposition~4.1]{2000},
a smooth quartic celestial surface in~$S^3$ is homeomorphic to
either a torus, a sphere or the disjoint union of two spheres
(see also \citep[Figure~12]{2012web}).
The following corollary classifies, up to homeomorphism, great celestial surfaces in~$S^3$ of any degree
and without assuming smoothness.

\begin{corollary}
\label{cor:top}
If $Z\subset S^3$ is a great celestial surface, then
$Z$ is homeomorphic to one of the following five normal forms (see \Cref{tab:tnf}):
\begin{Mlist}
\item the sphere~$S^2$,

\item the torus~$S^1\times S^1$,

\item
the union $S\cup S'$ of spheres~$S$ and~$S'$ \st $|S\cap S'|=2$,

\item
a union $T\cup T'$ of tori $T$ and $T'$
\st $T\cap T'$ is a cycle that is non-trivial in both $T$ and $T'$, or

\item
the disjoint union of a torus and a cycle.
\end{Mlist}
\end{corollary}
\begin{table}[!ht]
\caption{Topological normal forms.}
\label{tab:tnf}
\centering
\csep{2mm}
\begin{tabular}{ccccc}
\begin{tikzpicture}[scale=0.8] % sphere
\draw[thick] (0,0) arc (0:360:1 and 1);
\draw[densely dotted] (0,0) arc (0:180:1 and 0.3);
\draw[] (0,0) arc (0:-180:1 and 0.3);
\end{tikzpicture}
&
\begin{tikzpicture}[xscale=0.5,yscale=0.8] % torus
\centerarc[thick](0,0)(0:360:2:1)
\centerarc[thick](0,0)(-190:10:1:0.4)
\centerarc[thick](0,-0.4)(160:20:0.8:0.4)
\end{tikzpicture}
&
\begin{tikzpicture}[yscale=0.5,xscale=0.9, rotate=90] % doubly pinched torus
\draw[thick] (0,0) arc (270:-90:1.9 and 1);
\draw[thick] (0,0) arc (270:-90:0.5 and 1);
\draw[densely dotted] (-0.5,1) arc (0: 180:0.7 and 0.3);
\draw[]        (-0.5,1) arc (0:-180:0.7 and 0.3);
\draw[densely dotted] (1.9,1) arc (0: 180:0.7 and 0.3);
\draw[]        (1.9,1) arc (0:-180:0.7 and 0.3);
\draw[fill=red,red] (0,2) circle [radius=0.08];
\draw[fill=red,red] (0,0) circle [radius=0.08];
\end{tikzpicture}
&
\begin{tikzpicture}[xscale=0.3,yscale=0.3] % villarceau linked self-intersection
%\draw[help lines] (-4,-4) grid (4,4);
\draw (-4,0) to [out=90,in=90] (-3,0) to [out=90,in=90] (-2,0);
\draw[dotted] (-4,0) to [out=270,in=270] (-3,0) to [out=270,in=270] (-2,0);
\draw[thick] (-2,0) to [out=90, in=180] (0,1);
\draw[thick,red] (-3,0) to [out=90, in=180] (0,2);
\draw[thick] (-4,0) to [out=90, in=180] (0,3);
\draw[thick] (-2,0) to [out=270, in=180] (0,-1) to [out=0, in=190] (0.95,-0.9);
\draw[dotted] (0.95,-0.9) to [out=10, in=270] (2,0);
\draw[thick] (-4,0) to [out=270, in=180] (0,-3);
\draw[thick] (0,1) to [out=0, in=90] (2,0);
\draw[thick] (0,2) to [out=0, in=90] (3,0);
\draw[dotted] (0,-2) to [out=0,in=190] (1.6,-1.8);
\draw[thick] (1.65,-1.8) to [out=15,in=270] (3,0);
\draw[dotted] (0,2) to [out=0,in=100] (0.9,0.9);
\draw[thick] (0.9,0.9) to [out=280,in=90] (1,0) to [out=270,in=0] (0,-2);
\draw[thick] (0,3) to [out=0,in=100] (1.8,1.65);
\draw[dotted] (1.8,1.65) to [out=280,in=90] (2,0);
\draw[thick] (2,0) to [out=270,in=0] (0,-3);
\draw[thick,red] (-3,0) to [out=270, in=180] (0,-2);
\draw[thick,red,densely dotted] (0,2) to [out=0,in=90] (2,0) to [out=270, in=0] (0,-2);
\end{tikzpicture}
&
\begin{tikzpicture}[xscale=0.5,yscale=0.8]
\centerarc[thick](0,0)(0:360:2:1)
\centerarc[thick](0,-0.4)(160:20:0.8:0.4)
\centerarc[thick,red](-1.5,0)(-10:290:1.6:0.8)
\centerarc[very thick,red,dotted](-1.5,0)(-10:-65:1.6:0.8)
\centerarc[thick](0,0)(-190:10:1:0.4)
\end{tikzpicture}
\end{tabular}
\end{table}

We identify $S^3$ with the unit-quaternions
and let $\star$ denote the Hamiltonian product.
We define $A\star B$ for curves~$A, B\subset S^3$ as the Zariski closure of
\[
\set{a\star b \in S^3}{ a\in A \text{ and } b\in B}.
\]
We call $Z\subset S^3$ a \df{Clifford torus} if there exist
great circles $A,B\subset S^3$ \st $Z=A\star B$.
Clifford tori are of degree four and were pioneered by
William Kingdon Clifford (1845--1879) (see \citep[VII]{1998}).

If $Z\subset S^3$ is a celestial surface \st $\deg Z\neq 4$, then
Skopenkov and Krasauskas show in \citep[Main~Theorem~1.1]{2019}
that there exists circles~$A$ and~$B$ \st $Z$ is M\"obius equivalent to either $A\star B$,
or the Zariski closure of
\[
\mu^{-1}(\set{a+b}{a\in A,b\in B}).
\]
The method in \cite{2019} concerns the factorization of certain
bivariate polynomials with quaternionic coefficients.
\Cref{cor:prod} confirms and strengthens~\cite{2019}
under the additional assumption that $Z$ is great.
The alternative proof in this article is complementary in that it exposes the
self-intersections of such surfaces and does not depend on
the algebra of quaternionic polynomials.

\begin{corollary}
\label{cor:prod}
If a surface $Z\subset S^3$ contains a great and small circle through a general point,
then either
$Z=A\star B$ for some circles $A,B\subset S^3$ or $Z$ contains five concurrent circles.
\end{corollary}

Corollaries~\ref{cor:top} and \ref{cor:prod} consider properties of great celestial
surfaces in~$S^3$ that are invariant under homeomorphisms
and isometries of $S^3$, \resp.
The shapes in \Cref{def:shape} below are invariant under isometries of~$S^3$ as well
and thus are invariants in elliptic geometry (see \Cref{sec:E}).

A circle $C\subset Z$ in a surface $Z\subset S^3$ is
a \df{double circle} if $p_\infty\in C$ implies that $\deg\mu(Z)=\deg Z-2$.

\begin{definition}
\label{def:shape}
Suppose that a surface $Z\subset S^3$ satisfies the following conditions:
\begin{Menum}{C\arabic*.}{C\arabic*}
\item\label{C1} $Z$ is of degree $8$ and 2-circled,
\item\label{C2} $Z$ a disjoint union of great circles,
\item\label{C3} $Z$ a union of small circles, and
\item\label{C4} the singular locus of $Z$ consists of a great double circle.
\end{Menum}
We say that $Z$ has \df{Shape I}, \df{Shape II} or \df{Shape III} if
in addition to Conditions~\ref{C1}--\ref{C4} the following item~I, II and~III hold, \resp:
\begin{Menum}{\Roman*.}{\Roman*}
\item\label{I}
There exists tori $T$ and $T'$ in~$S^3$
\st $Z=T\cup T'$ and $T\cap T'$ is the great double circle in~$Z$.
If $C\subset Z$ is a great circle, then either $C\subset T$ or $C\subset T'$.
If $C\subset Z$ is a small circle, then $|C\cap T\cap T'|=2$.

\item\label{II}
The surface $Z$ is a torus and
a small circle in~$Z$ meets the great double circle tangentially at one point.

\item\label{III}
There exists a torus $T$ and a great circle $C\subset S^3$
\st $Z=T\,\cup\, C$ and $T\,\cap\, C=\varnothing$.% and $C$ and $T$ are linked.
\END
\end{Menum}
\end{definition}

See \Cref{fig:shape} for examples of each of the three shapes.
\Cref{fig:Iab} depicts two stereographic projections of a surface that has Shape~\ref{I},
\st exterior and interior of the orange topological torus component are interchanged.
\begin{figure}[!ht]
\centering
\csep{10mm}
\begin{tabular}{cc}
\fig{3}{4}{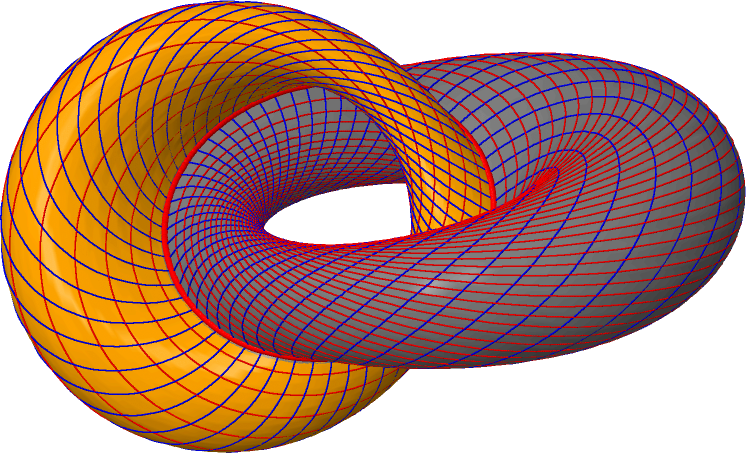} &
\fig{3}{4}{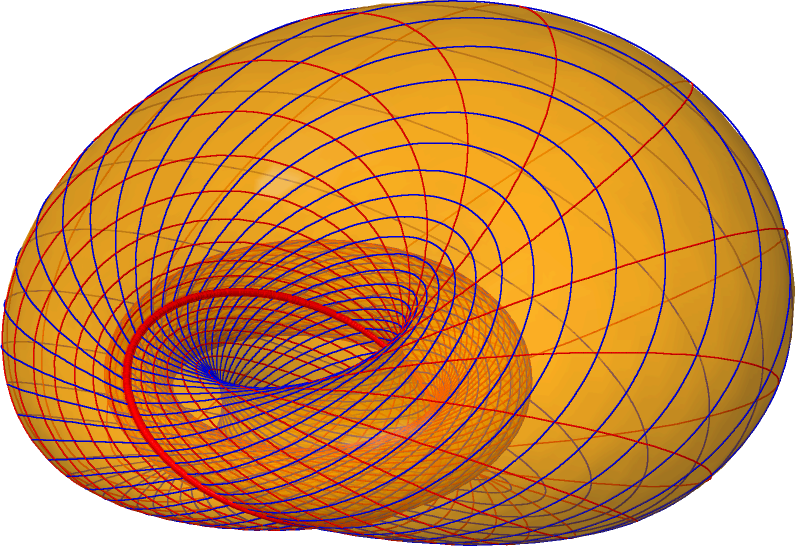}
\end{tabular}
\caption{Two stereographic projections of a surface in~$S^3$ that has Shape~\ref{I}.}
\label{fig:Iab}
\end{figure}

A smooth 4-circled, 5-circled or 6-circled surface in~$S^3$
is called a \df{ring cyclide}, \df{Perseus cyclide} and \df{Blum cyclide}, \resp.
By \citep[Proposition~39]{2024lt}, the Clifford tori are exactly the great ring cyclides.

\begin{corollary}
\label{cor:AB}
If $A,B\subset S^3$ are different circles \st either $A$ or $B$ is great,
then $A\star B$ is either a great ring cyclide or has Shape~\ref{I}, Shape~\ref{II} or Shape~\ref{III}.
\end{corollary}

We call surface $Z\subset S^3$ a \df{CO cyclide} or \df{EO cyclide}
if $Z$ stereographically projects to a circular cone and elliptic cone, \resp.
The CO~cyclides are also known as \df{spindle cyclides}.

%We are now ready to state the main result of this article:
\begin{theorem}
\label{thm:shape}
Suppose that $Z\subset S^3$ is a great celestial surface.
\begin{claims}
\item\label{thm:shape:a}
$\deg Z\in\{2,4,8\}$.
\item\label{thm:shape:b}
If $\deg Z=4$, then $Z$ is either a
Blum cyclide, Perseus cyclide, ring cyclide, EO~cyclide, or CO cyclide.
Moreover, $Z=A\star B$ for circles $A, B\subset S^3$ if and only if $Z$ is a ring cyclide.
\item\label{thm:shape:c}
If $\deg Z=8$, then
$Z$ has either Shape~\ref{I}, Shape~\ref{II} or Shape~\ref{III}.
Moreover, there exists
a great circle $A\subset S^3$ and small circle $B\subset S^3$ \st $Z\in \{ A\star B,B\star A\}$.
\end{claims}
\end{theorem}

\begin{example}
\label{exm:shape}
Let the great circle $A_0\subset S^3$
and the small circles $B_1, B_2, B_3\subset S^3$
be parametrized as follows:
\begin{Mlist}

\item[$A_0$] $:= \set{(\cos\alpha,\sin\alpha,0,0)}{0\leq\alpha<2\pi}$,

\item[$B_1$] $:=\left\{
\left(
\frac{12+8\cos\beta}{17+12\cos\beta},~
\frac{8\sin\beta}{17+12\cos\beta},~
0,~
\frac{9+12\cos\beta}{17+12\cos\beta}
\right)
~:~
0\leq \beta<2\pi
\right\}$,

\item[$B_2$] $:=\left\{
\left(
\frac{2+\cos\beta}{3+2\cos\beta},~
\frac{\sin\beta}{3+2\cos\beta},~
0,~
\frac{2+2\cos\beta}{3+2\cos\beta}
\right)
~:~
0\leq \beta<2\pi
\right\}$, and

\item[$B_3$] $:=\left\{
\left(
\frac{6+2\cos\beta}{11+6\cos\beta},~
\frac{2\sin\beta}{11+6\cos\beta},~
0,~
\frac{9+6\cos\beta}{11+6\cos\beta}
\right)
~:~
0\leq \beta<2\pi
\right\}$.

\end{Mlist}
The surfaces
$A_0\star B_1$,
$A_0\star B_2$ and
$A_0\star B_3$
have Shape~\ref{I}, Shape~\ref{II} and Shape~\ref{III},
\resp~(see \citep[\texttt{orbital}]{2017} for a
\href{https://github.com/niels-lubbes/orbital#example-3-computing-products-of-circles}{verification}).
The stereographic projections of these surfaces
are depicted in \Cref{fig:shape}.
\END
\end{example}

\begin{remark}
\Cref{thm:shape}\ref{thm:shape:c} is the main contribution of the current article.
Theorems~\ref{thm:shape}\ref{thm:shape:a} and \ref{thm:shape}\ref{thm:shape:b}
follow from
\citep[Theorem~1]{2021circle} and
\citep[Theorem~1(c) and Corollary~4]{2024lt}, \resp.
We summarized \Cref{thm:shape} and its corollaries in \Cref{tab:shape}.
\END
\end{remark}

\begin{table}[!ht]
\centering
\caption{
By \Cref{thm:shape} and Corollaries~\ref{cor:top}, \ref{cor:prod} and \ref{cor:AB}
a great $\lambda$-circled celestial surface of degree~$d$ in~$S^3$ is
for some circles $A,B\subset S^3$ characterized by a row.}
\label{tab:shape}
\csep{6mm}
\begin{tabular}{llllc}
$d$ & $\lambda$ & name or shape   & topological type & $A\star B$ \\\hline
$2$ & $\infty$  & 2-sphere        & sphere           & no         \\\hdashline
$4$ & $6$       & Blum cyclide    & torus            & no         \\
$4$ & $5$       & Perseus cyclide & torus            & no         \\
$4$ & $4$       & ring cyclide    & torus            & yes        \\
$4$ & $3$       & EO cyclide      & two spheres      & no         \\
$4$ & $2$       & CO cyclide      & two spheres      & no         \\\hdashline
$8$ & $2$       & Shape \ref{I}   & two tori         & yes        \\
$8$ & $2$       & Shape \ref{II}  & torus            & yes        \\
$8$ & $2$       & Shape \ref{III} & torus and cycle  & yes        \\
\end{tabular}
\end{table}

\Cref{thm:shape}\ref{thm:shape:c} suggests the following conjecture:

\begin{conjecture}
\label{cnj}
If $Z\subset S^3$ is a great celestial surface \st $\deg Z\neq\{2,4\}$,
then there exists a diffeomorphism $f\c S^3\to S^3$ \st $f(Z)$
is equal to either
$A_0\star B_1$,
$A_0\star B_2$ or
$A_0\star B_3$.
\end{conjecture}

% \begin{remark}[open problem]
% The classification
% of surfaces~$Z\subset \R^n$ containing a line an circle through each point
% up to diffeomorphisms of $\R^n$ is an open problem when $n>3$.
% If $n\leq 3$, then $Z$ is either a plane or a quadratic surface \cite{2013}
% and if $n>3$, then $Z$ is smooth and of degree at most $n-1\leq 4$ \citep[Corollary~4]{2021circle}.
% \END
% \end{remark}

{\bf Overview.}
In \Cref{sec:E}, we setup a projective model for elliptic geometry.
In \Cref{sec:delta}, we introduce an invariant
for curve components in the singular loci of projective surfaces.
In \Cref{sec:div}, we use divisor classes to characterize the incidences between
circles and complex double lines in great celestial surfaces $Z\subset S^3$
of degree eight.
In \Cref{sec:central}, we use this invariant in combination with the central projection
to obtain a characterization of the incidences between
circles and double curves in~$Z$.
This characterization is used in
\Cref{sec:shape} to prove the main result~\Cref{thm:shape}.
\ifarxiv
\makeatletter
\renewcommand{\@cftmaketoctitle}{} % removes Contents title
\makeatother
\begingroup
\vspace{1mm} % spacing between section header
\def\addvspace#1{\vspace{-1.5mm}} % spacing between section titles in overview
\tableofcontents
\endgroup
\fi

\section{Projective model for elliptic geometry}
\label{sec:E}

In order to prove \Cref{thm:shape},
we investigate curves at complex infinity.
To uncover these hidden curves we
define a \df{real variety} $X$ to be a complex irreducible variety together with
an antiholomorphic involution $\sigma_X\c X\to X$
called the \df{real structure} of $X$ (see \citep[Section~I.1]{1989}).
We denote its real points by
\[
X_\R:=\set{p\in X}{\sigma_X(p)=p}.
\]
Such varieties can always be defined by polynomials with real coefficients \citep[Section~6.1]{1991}.

In what follows, points, curves, surfaces and projective spaces $\P^n$ are real algebraic varieties
and maps between such varieties
are compatible with their real structures
unless explicitly stated otherwise.
Conics are real and reduced by default, but may be reducible.
By default, we assume that the real structure $\sigma_{\P^n}\c\P^n\to\P^n$ sends $x$
to $(\overline{x_0}:\ldots:\overline{x_n})$, where $\overline{\,\cdot\,}$ denotes the complex conjugate.

As circles play a central role, it is natural to consider
the \df{M\"obius quadric} for our space:
\[
\S^3:=\set{x\in\P^4}{-x_0^2+x_1^2+x_2^2+x_3^2+x_4^2=0}.
\]
The \df{elliptic absolute} is
defined as the following hyperplane section of~$\S^3$ without real points:
\[
\E:=\set{x\in \S^3}{ x_0=0}.
\]
Let $\_\hstar\_\c\S^3\times\S^3\dto \S^3$ be the rational map defined by
\begin{align*}
(x,y)\mapsto (x_0y_0:& ~x_1y_1-x_2y_2-x_3y_3-x_4y_4:x_1y_2+x_2y_1+x_3y_4-x_4y_3:\\
                     & ~x_1y_3-x_2y_4+x_3y_1+x_4y_2:x_1y_4+x_2y_3-x_3y_2+x_4y_1).
\end{align*}
We consider the following complex transformations of~$\S^3$,
where $\aut_\C\P^4$ denotes the complex projective transformations of~$\P^4$:
\begin{align*}
\aut_\C\S^3:=&\set{\varphi\in\aut_\C\P^4}{\varphi(\S^3)=\S^3},\\
\aut_\E\S^3:=&\set{\varphi\in\aut_\C\S^3}{\varphi(\E)=\E},\\
\lt\S^3:=&\set{\varphi\c\S^3\dto\S^3}{\varphi(x)=p\,\hstar\,x,~ p\in\S^3\setminus\E}, \text{ and}\\
\rt\S^3:=&\set{\varphi\c\S^3\dto\S^3}{\varphi(x)=x\,\hstar\,p,~ p\in\S^3\setminus\E}.
\end{align*}
The \df{M\"obius transformations} are defined as
\[
\aut\S^3:=\set{\varphi\in\aut_\C\S^3}{\varphi\circ\sigma_{\P^4}=\sigma_{\P^4}\circ\varphi}.
\]
The
\df{elliptic transformations},
\df{left Clifford translations} and
\df{right Clifford translations} are defined
as
\[
\aut_\E\S^3\cap\aut\S^3\quad
\lt\S^3\cap\aut\S^3\quad\text{and}\quad
\rt\S^3\cap\aut\S^3,
\quad\text{\resp.}
\]
The \df{left generator} and \df{right generator}
that pass through $p\in\E$ are defined as
\[
\cL_p:=\set{q\,\hstar\,p}{q\in\S^3\setminus\E}
\quad\text{and}\quad
\cR_p:=\set{p\,\hstar\,q}{q\in\S^3\setminus\E},
\quad\text{\resp.}
\]

Let $\bR\c \S^n_\R\to S^n$ denote the isomorphism that sends $x$ to
$\left({x_1}/{x_0},\ldots,{x_{4}}/{x_0}\right)$.
If $V\subset \S^3$ is a variety, then we define
\[
V(\R):=\bR(V_\R).
\]
Notice that $\S^3(\R)=S^3$.

The elliptic transformations correspond to rotations and reflections of~$S^3$
and the left/right Clifford translations correspond to isoclinic rotations of~$S^3$.
The following proposition is classical (see \citep[\textsection7.9 and 7.93]{1998}).
Recall that $\_\star\_\c S^3\times S^3\to S^3$ denotes the Hamiltonian product, where
we identified $S^3$ with the unit quaternions.

\begin{proposition}
\label{prp:E}~
\begin{claims}
\item\label{prp:E:a}
$\bR(x\,\hstar\,y)=\bR(x)\star\bR(y)$ for all $x,y\in\S^3_\R$.
\item\label{prp:E:b}
$\lt\S^3,\rt\S^3\subset\aut_\E\S^3$.
\item\label{prp:E:c}
For all $p\in\E$,
the generators $\cL_p$ and $\cR_p$ are the two complex lines in $\E$ containing~$p$.
\item\label{prp:E:d}
For all $\varphi\in\lt\S^3$, we have $\varphi(\cL_p)=\cL_p$.
\\For all $\varphi\in\rt\S^3$, we have $\varphi(\cR_p)=\cR_p$.
\end{claims}
\end{proposition}

\begin{proof}
See \citep[Proposition~4]{2024lt}.
\end{proof}

Notice that the complex conjugate of a left (right) generator is again a left (right) generator.

\begin{definition}
\label{def:model}
We call a conic $C\subset \S^3$ a \df{great circle} or \df{small circle},
if $C(\R)$ is as such in~$S^3$.
Similarly, a surface $X\subset\S^3$ is called
\df{$\lambda$-circled},
\df{great} or \df{celestial}
% \df{CO cyclide},
% \df{EO cyclide},
% \df{ring cyclide},
% \df{Perseus cyclide} or
% \df{Blum cyclide}
if $X(\R)\subset S^3$ is defined as such in \Cref{sec:intro}.
\END
\end{definition}

\section{Sectional delta invariant}
\label{sec:delta}

In this section, we introduce an invariant for 1-dimensional
components in the singular locus of a surface.
This invariant measures how singular such a component is.

\begin{definition}
\label{def:di}
The \df{delta invariant} of a point $p$ in the curve $C\subset\P^n$
with structure sheaf~$\cO$
is defined as $\delta_p(C):=\operatorname{length}(\widetilde{\cO_p}/\cO_p)$,
where $\widetilde{\cO_p}$ denotes the
integral closure of the stalk~$\cO_p$ (see \citep[Exercise IV.1.8]{1977} or
\citep[\href{https://stacks.math.columbia.edu/tag/0C3Q}{Tag 0C3Q}]{2018stacks}).
\END
\end{definition}

\begin{remark}
Notice that the delta invariant of a singular point in a curve is a non-zero positive integer.
Informally, we may think of~$\delta_p(C)$
as the number of double points that are concentrated at $p$ (see \citep[page 85]{1968}).
\END
\end{remark}

We denote the \df{singular locus} of a complex surface~$X\subset\P^n$
by~$\Sing X$.
If $C\subset X$ is a complex curve, then
$p_a(C)$ denotes its \df{arithmetic genus}
and
$p_g(C)$ denotes its \df{geometric genus}.

\begin{lemma}
\label{lem:genus}
Suppose that $X\subset\P^n$ is a complex surface and
$H\subset X$ a general hyperplane section.
\begin{claims}
\item\label{lem:genus:a}
The complex curve $H$ is irreducible and reduced, and $\Sing H=H\cap\Sing X$.

\item\label{lem:genus:b}
$p_a(H)-p_g(H)=\sum_{p\in\Sing H}\delta_p(H)$.

\item\label{lem:genus:c}
For all irreducible hyperplane sections~$H'\subset X$,
we have $p_a(H)=p_a(H')$ and $p_g(H)\geq p_g(H')\geq 0$.
Moreover, if $H'$ is also general, then $p_g(H)=p_g(H')$.
\end{claims}
\end{lemma}

\begin{proof}
\ref{lem:genus:a}
The first assertion follows from the Bertini theorems at~\citep[\href{https://stacks.math.columbia.edu/tag/0G4C}{Tag 0G4C}]{2018stacks}
and the second assertion follows from \citep[Theorem~17.16]{1992}.

\ref{lem:genus:b}
This is the genus formula at \citep[Exercise~IV.1.8a]{1977} (see also \citep[Section~2.4.6]{2016}).

\ref{lem:genus:c}
It follows from \citep[Exercise~V.1.3]{1977} that $p_a(H)$ only depends
on the linear equivalence class of~$H$. A non-general hyperplane section~$H'\subset X$
maybe more singular than $H$ and thus $\sum_{p\in\Sing H}\delta_p(H)\leq \sum_{p\in\Sing H'}\delta_p(H')$
so that $p_g(H)\geq p_g(H')\geq 0$ by \ASN{lem:genus:b}.
If $H'$ is also general, then $p_g(H)=p_g(H')$ as a straightforward consequence of
\ASN{lem:genus:b} and \Cref{def:di}.
\end{proof}

Suppose that $X\subset\P^n$ is a complex surface.
The \df{sectional arithmetic genus}~$a(X)$
and the \df{sectional geometric genus}~$g(X)$
are defined
as the arithmetic and geometric genus of a general hyperplane section of $X$, \resp.
The \df{total delta invariant} of~$X$ is defined as $\delta(X):=a(X)-g(X)$.
These invariants are well defined by \Cref{lem:genus}.

\begin{definition}
\label{def:delta}
Let $X\subset\P^n$ be a complex surface
\st the 1-dimensional part of the singular locus of~$X$
admits the following decomposition
into irreducible complex curve components:
$
C_1\cup\cdots\cup C_r.
$
A \df{sectional delta invariant} for $X$ is a function
\[
\Delta_X\c \set{C_i}{1\leq i\leq r} \to \Z_{>0}
\]
that satisfies the following
axioms for all $1\leq i\leq r$, real structures $\sigma\c X\to X$,
and complex projective automorphisms~$\alpha\in \aut\P^n$:
\begin{Menum}{A\arabic*.}{A\arabic*}
\item\label{a1}
$\Delta_X(C_1)+\cdots+\Delta_X(C_r)=\delta(X)$.

\item\label{a2}
$\Delta_X( C_i)\geq \deg C_i$.

\item\label{a3}
$\Delta_X(C_i)=\Delta_X(\sigma(C_i))$ and $\Delta_X(C_i)=\Delta_{\alpha(X)}(\alpha(C_i))$.

\item\label{a4}
If $\rho\c X\dto Z\subset\P^m$ is a complex birational linear map
\st $g(X)=g(Z)$ and
$\rho|_{C_i}\c C_i\dto \rho(C_i)$ is birational, then
\[
{\Delta_Z(\rho(C_i))}\cdot {\deg C_i}={\Delta_X(C_i)}\cdot{\deg\rho(C_i)}.
\]

\item\label{a5}
If $\rho\c X\to Z\subset\P^m$ is a generically finite $q:1$ linear morphism
\st
$\deg \rho(C_i)=(\deg C_i)/q$, then
\[
\Delta_Z(\rho(C_i))\cdot\deg C_i\geq\Delta_X(C_i)\cdot\deg\rho(C_i).
\]
\end{Menum}
We write $\Delta(C)$ instead of~$\Delta_X(C)$
if it is clear from the context that $C\subset X$.
\END
\end{definition}

\begin{proposition}
\label{prp:delta}
If $X\subset \P^n$ is a complex surface, then
there exists a sectional delta invariant~$\Delta_X$.
\end{proposition}

\begin{proof}
Let $\cH$ denote the set of irreducible and reduced hyperplane sections of~$X$.
Suppose that $H\subset X$ is a general hyperplane section of $X$.
We assume the notation at \Cref{def:delta}
and
consider the following functions
for all $1\leq i\leq r$ and $W\in \cH$:
\[
\Delta_X(C_i,W):=\sum_{p\in W\cap C_i}\delta_p(W)
\qquad\text{and}\qquad
\Delta_X(C_i):=\Delta_X(C_i,H).
\]
It follows from \Cref{lem:genus}\ref{lem:genus:c} that
the sectional delta invariant~$\Delta_X(C_i)$ does not depend on the
choice of a general hyperplane section~$H$ and is therefore well-defined.
Since $\Sing H=(C_1\cup\cdots\cup C_r)\cap H$ by \Cref{lem:genus}\ref{lem:genus:a},
\AXM{a1} is a direct consequence of \Cref{lem:genus}\ref{lem:genus:b}.
\AXM{a2} is a direct consequence of B\'ezout's theorem.

{\bf Claim 1.}~{\it
For all $p,q\in C_i\cap H$, $1\leq i\leq r$, real structures~$\sigma\c X\to X$
and complex projective automorphisms~$\alpha\in \aut\P^n$ we have}
\[
\delta_p(H)=\delta_q(H)=\delta_{\sigma(p)}(\sigma(H))=\delta_{\alpha(p)}(\alpha(H)).
\]
Let $t_0\,x_0+\cdots+t_n\,x_n$ with $t\in\P^n$ be the defining polynomial of the complex hyperplane~$H$.
Notice that the ideal of~$C_i\subset\P^n$ is generated by polynomial forms in~$\C[x]=\C[x_0,\ldots,x_n]$.
Instead over the coefficient field $\C$, let us consider the defining polynomials
of the complex curve~$C_i$ and the hyperplane over the function field~$\C(t)=\C(t_0,\ldots,t_n)$.
As we do not choose any particular value for $t\in\P^n$, we ensure that $H$ is general.
We consider a field extension $E/\C(t)$ \st its Galois group acts transitively on the elements in $H\cap C_i$.
The delta invariant is an algebraic invariant in the sense that
it can be computed in the ring~$E[x]$ from the defining polynomials.
Hence, all algebraic invariants (and in particular the delta invariant) are
the same for all elements in~$C_i\cap H$.
Since both $\sigma$ and $\alpha$ act on~$E[x]$ as automorphisms,
we find that an algebraic invariant of $p$ is equal to $\sigma(p)$ and $\alpha(p)$ for all $p\in X$.
This concludes the proof for Claim~1.

% It follows from Claim~1 and \Cref{lem:genus}\ref{lem:genus:c} that $\Delta_X(C_i)\leq \Delta_X(C_i,W)$
% for all $W\in\cH$ and $1\leq i\leq r$.
% Hence, by \AXM{a1}, the sectional delta invariant~$\Delta_X(C_i)$ does not depend on the
% choice of a general hyperplane section~$H$ and is therefore well-defined.

\AXM{a3} is a direct consequence of Claim~1.

We now proceed with the proof for \AXM{a4}.
Let $H'\subset Z$ be a general hyperplane section and let $W:=\rho^{-1}(H')$.
We deduce from \Cref{lem:genus}\ref{lem:genus:a}
that $H'$ must be irreducible and reduced, which implies that $W\in\cH$.
However, $W$ is itself not necessarily general
as the hyperplane spanned by~$W$ passes through the center of the linear projection~$\rho$.
Let $1\leq i\leq r$ and $U_p\subset X$ be an arbitrary small complex analytic neighborhood of~$p\in W\cap C_i$.
Since $\rho$ is birational, it is defined at $W\cap C_i$ and thus restricts to
a complex analytic isomorphism $U_p\to\rho(U_p)$.
The delta invariant $\delta_p(W)$ is a complex analytic invariant by \citep[Exercise~IV.1.8c]{1977}
and thus
\[
\delta_p(W)=\delta_{\rho(p)}(H').
\]
We deduce from Claim~1 applied to $H'\subset Z$ and B\'ezout's theorem that
\[
\Delta_Z(\rho(C_i))=\deg(\rho(C_i))\cdot \delta_{\rho(p)}(H').
\]
Again by B\'ezout's theorem it follows that
\[
\Delta_X(C_i,W)=\deg(C_i)\cdot \delta_{p}(W).
\]
We know from \Cref{lem:genus}\ref{lem:genus:c} that $\Delta_X(C_i)\leq \Delta_X(C_i,W)$.
Hence, we established that
\[
\frac{\Delta_Z(\rho(C_i))}{\deg\rho(C_i)}
=
\frac{\Delta_X(C_i,W)}{\deg C_i}
\geq
\frac{\Delta_X(C_i)}{\deg C_i}.
\]
However, since $g(X)=g(Z)$ by assumption, we have
the equality~$\Delta_X(C_i,W)= \Delta_X(C_i)$
and thus we conclude that \AXM{a4} holds.

The proof of \AXM{a5} follows the proof of \AXM{a4}.
Again the finite morphism~$\rho$ is defined at $p\in W\cap C_i$.
In this case however, $\rho$ defines
a complex analytic isomorphism~$U_{p'}\to\rho(U_{p'})$ for each point~$p'$ in the fibre~$(\rho^{-1}\circ\rho)(p)$.
In other words, the $q:1$~covering $\rho$ defines locally a complex analytic ismorphism on each of its $q$ sheets.
The remaining arguments are the same as for \AXM{a4}, except we do not need to prove the equality.
This concludes the proof for the only remaining \AXM{a5}.
\end{proof}

\begin{remark}
We will not use \AXM{a4} in this article, but
we believe that the notion of sectional delta invariant in~\Cref{def:delta}
is of interest outside the scope of this article.
We conjecture that the $g(X)=g(Z)$ assumption in \AXM{a4} can be
omitted and that the inequality in \AXM{a5} can be replaced by an equality.
\END
\end{remark}

\section{Divisor classes of curves on singular surfaces}
\label{sec:div}

In this section, we characterize divisor classes of
complex curves on the singular surfaces in~$\P^n$
that will appear in \textsection\ref{sec:central}.
We also compute the total delta invariants of these surfaces.

A \df{smooth model} of a surface $X\subset\P^n$ is a birational morphism $\varphi\c Y\to X$
from a nonsingular surface~$Y$, that does not contract complex $(-1)$-curves.
See \citep[Theorem~2.16]{2007kol} for the existence and uniqueness of the smooth model
up to biregular isomorphisms.

The \df{N\'eron-Severi lattice}~$N(X)$ is
an additive group  defined by the divisor classes on~$Y$ up to numerical equivalence.
This group comes with a unimodular intersection product~$\cdot$
and a unimodular involution~$\sigma_{X*}\c N(X)\to N(X)$ induced by the real structure~$\sigma_X\c X\to X$.
By default, $\sigma_{X*}$ is the identity map~$\id\c N(X)\to N(X)$.

Suppose that $C\subset X$ is a complex and possibly reducible curve.
Let $C_Y\subset Y$ denote the union of complex curves
in~$\varphi^{-1}(C)$ that are not contracted to complex points by the morphism~$\varphi$.
The \df{class}~$[C]\in N(X)$ of $C$ is defined as the divisor class of~$C_Y$.

The \df{class of hyperplane sections} $\h\in N(X)$ is defined as the
class of a general hyperplane section of $X$.
The \df{canonical class} $\k\in N(X)$ is defined as the canonical class of the smooth model~$Y$
(see \citep[Example~II.8.20.3]{1977} or \citep[\textsection1.4.1]{2016}).

\begin{remark}
\label{rmk:N}
If $C,D\subset X$ are complex irreducible curves \st $[C]\cdot [D]>0$,
then $|C\cap D|>0$. However, we cannot assume
that $|C\cap D|=[C]\cdot [D]$.
In particular, if $\varphi^{-1}(C)$ and $\varphi^{-1}(D)$ are disjoint in~$Y$,
then $C$ and $D$ may still meet at the singular locus of~$X$.
\END
\end{remark}

A surface $X\subset\P^n$ has \df{lattice type~$(\alpha,\beta,\gamma)\in\Z^3_{\geq 0}$} if
there exists a smooth model $\P^1\times\P^1\to X$ and
\begin{Mlist}
\item $N(X)\cong \bas{\l_0,\l_1}_\Z$ with $\l_0^2=\l_1^2=0$ and $\l_0\cdot\l_1=1$,
\item $\k=-2\,\l_0-2\,\l_1$ is the canonical class,
\item $\h=\alpha\,\l_0+\beta\,\l_1$ is the class of hyperplane sections, and
\item the total delta invariant $\delta(X)$ is equal to~$\gamma$.
\end{Mlist}

\begin{remark}
\label{rmk:lt}
If $C\subset X$ is a complex curve that is not contained in the singular locus of~$X$,
then $\deg C=\h\cdot [C]$.
In particular, we observe the following:
\begin{Mlist}
\item If $X$ has lattice type~$(2,2,8)$ and $\deg C\leq 2$, then $\deg C=2$ and $[C]\in\{\l_0,\l_1\}$.
\item If $X$ has lattice type~$(2,1,3)$ and $\deg C=1$, then $[C]=\l_0$.
\item If $X$ has lattice type~$(2,1,3)$ and $\deg C=2$, then $[C]=\l_1$. \END
\end{Mlist}
\end{remark}

\begin{proposition}
\label{prp:N2}
A smooth surface~$X\subset\P^n$ of degree two has lattice type~$(1,1,0)$
and either $\sigma_{X*}(\l_0)=\l_1$ or $\sigma_{X*}=\id$.
\end{proposition}

\begin{proof}
See \citep[Examples~II.6.6.1 and~II.8.20.3]{1977}.
\end{proof}

\begin{lemma}
\label{lem:gg}
Suppose that $X\subset\P^n$ is a surface with smooth model~$\varphi\c Y\to X$,
canonical class~$\k$ and class of hyperplane sections~$\h$.
\begin{claims}
\item\label{lem:gg:a}
If $H\subset X$ is a general hyperplane section, then
$p_g(H)=\tfrac{1}{2}(\h^2+\h\cdot\k)+1$.

\item\label{lem:gg:b}
If $C\subset X$ is a complex curve \st $[C]^2+[C]\cdot\k<-2$,
then $\varphi^{-1}(C)$ contains at least two complex curves that are not contracted
via $\varphi$ to complex points.
\end{claims}
\end{lemma}

\begin{proof}
\ref{lem:gg:a}
Suppose that $\varphi\c Y\to X$ is a smooth model and let $D$ be the proper transform of $H$ along $\varphi$.
We have $p_a(D)=\tfrac{1}{2}(\h^2+\h\cdot\k)+1$ by the arithmetic genus formula \citep[Exercise~V.1.3]{1977}.
It follows from the Bertini theorem at~\citep[Corollary~10.9]{1977} that the general curve $D$ in the
linear series associated to the morphism~$\varphi$ is smooth, which implies that $p_a(D)=p_g(D)$.
As the geometric genus is a birational invariant, it follows that $p_g(D)=p_g(H)$ as asserted.

\ref{lem:gg:b}
If $\varphi^{-1}(C)$ is irreducible, then
$0\leq p_g(C)\leq \tfrac{1}{2}([C]^2+[C]\cdot\k)+1$
by the geometric genus formula
(see
\citep[\textsection2.4.6]{2016} or
\citep[Remark~IV.1.1.1 and Exercise~IV.1.8.a]{1977}).
\end{proof}

\begin{lemma}
\label{lem:delta}
Suppose that $X\subset\P^n$ is a surface of degree~$d$
with canonical class~$\k$ and class of hyperplane sections~$\h$.
\begin{claims}
\item\label{lem:delta:a}
If $X\subset\P^3$, then
$\delta(X)=\frac{d}{2}\,(d-4)-\frac{1}{2}\,\h\cdot\k$.

\item\label{lem:delta:b}
If $X\subset\S^3$, then
$\delta(X)=\frac{d}{2}\,(\frac{d}{2}-3)-\frac{1}{2}\,\h\cdot\k$.
\end{claims}
\end{lemma}

\begin{proof}
Suppose $H\subset X$ is a general hyperplane section so that $\h=[H]$ and $d=\h^2$.

\ref{lem:delta:a}
Since $H$ is a planar curve so that $a(X)=p_a(H)=\frac{1}{2}(d-1)(d-2)$ by \citep[Example~2.17]{2016},
we conclude from \Cref{lem:gg}\ref{lem:gg:a} that $\delta(X)$ is as asserted.

\ref{lem:delta:b}
We observe that $H$ is a complete intersection
curve of degree~$d$ that is contained in a two-sphere~$Q\subset\S^3$.
By \Cref{prp:N2}, we have $N(Q)\cong\bas{\l_0,\l_1}_\Z$, $\h_Q=\l_0+\l_1$ and $\k_Q=-2\,h_Q$.
Suppose that $[H]_Q$ is the class of $H$ in $N(Q)$
so that
\[
\h_Q\cdot [H]_Q=(\l_0+\l_1)\cdot (\alpha\,\l_0+\alpha\,\l_1)=2\,\alpha=d.
\]
We find that $\alpha=\frac{1}{2}\,d$ so that
$a(X)=p_a(H)=\frac{1}{2}([H]_Q^2+\k_Q\cdot[H]_Q)+1=\frac{1}{4}d^2-d+1$
by the arithmetic genus formula \citep[Exercise~V.1.3]{1977}.
Since $g(X)=p_g(H)=\frac{1}{2}(d+\h\cdot\k)+1$ by \Cref{lem:gg}\ref{lem:gg:a},
we conclude that $\delta(X)=a(X)-g(X)$ is as asserted.
\end{proof}

\begin{proposition}
\label{prp:smooth8}
If $X\subset \S^3$ celestial surface of degree eight,
then there exists a biregular isomorphism~$f\c \P^1\times\P^1\to X_N\subset \P^8$
and linear projection $\eta\c \P^8\dto \P^4$ \st
\begin{Mlist}
\item the components of $f$ form a basis for the vector space of bidegree~$(2,2)$ forms on~$\P^1\times\P^1$,
\item the restriction $\eta|_{X_N}\c X_N\to X$ is a birational morphism \st $\deg X_N=\deg X$, and
\item the composition $\eta\circ f\c \P^1\times\P^1\to X$ defines a smooth model for $X$.
\end{Mlist}
\end{proposition}

\begin{proof}
It follows from \citep[Theorem~11]{2001} that there exists a complex birational map $g\c\P^1\times\P^1\dto X$
whose components form a complex subspace of the vector space~$V$
of bidegree~$(2,2)$ forms on~$\P^1\times\P^1$.
The components of $f$ form a real basis of~$V$ so that $f$ is associated to the anticanonical class of~$\P^1\times\P^1$.
It follows from \citep[Theorem~8.3.2(iii)]{2012dol}
that $X_N=f(\P^1\times\P^1)$ is the anticanonical model of a
del Pezzo surface of degree~8
and via $f$ biregular isomorphic to~$\P^1\times\P^1$
($X_N$ is called a ``Veronese-Segre surface'' in \citep[\textsection8.4.1]{2012dol}).

Alternatively, we show that $\deg X_N=8$ and $f$ is a biregular isomorphism by choosing a basis for $V$ \st
\[
\begin{array}{@{}l@{}l@{}l@{}l@{}l@{}l@{}l@{}l@{}l@{}l@{}}
f(x,y)
&=(
x_0^2\,y_0^2   &:x_0^2\,y_0\,y_1    &:x_0^2\,y_1^2  &:
x_0\,x_1\,y_0^2&:x_0\,x_1\,y_0\,y_1 &:x_0\,x_1\,y_1^2 &:
x_1^2\,y_0^2   &:x_1^2\,y_0\,y_1    &:x_1^2\,y_1^2) \\
&=(z_0&:z_1&:z_2&:z_3&:z_4&:z_5&:z_6&:z_7&:z_8).
\end{array}
\]
We verify that for all $0\leq k\leq 5$ and
$
(i,j)\in\{(0,1),(1,2),(3,4),(4,5),(6,7),(7,8)\}
$, either
$(z_k,z_{k+3})=(0,0)$, $(z_i,z_j)=(0,0)$ or
$(x_0:x_1;y_0:y_1)=(z_k:z_{k+3};z_i:z_j)$.
Since $z_i\neq 0$ for some $0\leq i\leq 8$, it follows that $f$ a biregular isomorphism.
The pullback of two general hyperplane sections of $X_N$
are general bidegree $(2,2)$ forms and such forms intersect in~$8$ points.
This implies that $\deg X_N=8$.

As a direct consequence of the definitions, $g=\eta\circ f$ for some linear projection
$\eta\c \P^8\dto \P^4$ with linear variety $W\subset \P^8$ as projection center.
If $W\cap X_N\neq\varnothing$, then $\deg X<\deg X_N=8$ and thus we arrive at a contradiction.
Hence, $\eta\circ f$ is a morphism and since $\P^1\times\P^1$ does not
have complex $(-1)$-curves, we conclude that $\eta\circ f$ defines
a smooth model for~$X$.
\end{proof}

\begin{proposition}
\label{prp:N8}
A celestial surface~$X\subset\S^3$ of degree eight has lattice type $(2,2,8)$.
\end{proposition}

\begin{proof}
It follows from \Cref{prp:smooth8} and \citep[Example~II.8.20.3]{1977}
that $N(X)=\bas{\l_0,\l_1}_\Z$ and $-\k=\h=2\,\l_0+2\,\l_1$.
Hence, $\delta(X)=8$ by \Cref{lem:delta}\ref{lem:delta:b}.
\end{proof}

A \df{pencil} on a surface $X\subset\S^3$ is defined as
an irreducible hypersurface~$P\subset X\times \P^1$
\st the 1st and 2nd projections $\pi_1\c X\times \P^1\dto X$
and $\pi_2\c X\times \P^1\dto \P^1$ are dominant.
The \df{member}~$P_i\subset X$ for index $i\in \P^1$ of the pencil~$P$
is defined as
the Zariski closure of~$\pi_1(P\cap X\times\{i\})$.
We call a complex point $p\in X$ a \df{base point} of~$P$,
if $p\in P_i$ for all~$i\in\pi_2(P)$.
We call $P$ a \df{pencil of conics} if
$P_i$ is a complex irreducible conic for almost all~$i\in \P^1$.
We call $P$ a \df{pencil of circles} if it is a pencil of conics \st
$P_i$ is a circle for infinitely many~$i\in \P^1_\R$.

\begin{remark}
If $P\subset X\times \P^1$ is a pencil of conics on a celestial surface $X\subset \S^3$,
then it follows from \citep[Theorem~9]{2001} that
$P$ is the Zariski closure of the graph of a rational map~$X\dto\P^1$ whose fibers are complex conics.
This implies that the first projection~$\pi_1$ is birational.
\END
\end{remark}

\begin{proposition}
\label{prp:P8}
Suppose that $X\subset\S^3$ is a celestial surface of degree eight.
\begin{claims}
\item\label{prp:P8:a} The surface $X$ is 2-circled and the two pencils of circles
that cover $X$ are base point free.
\item\label{prp:P8:b} If $L\subset X$ is a complex line, then $L\subset\Sing X$ and $L$ is non-real.
\item\label{prp:P8:c} The singular locus of~$X$ does not contain isolated singularities.
\end{claims}
\end{proposition}

\begin{proof}
Let $f\c \P^1\times\P^1\to X_N\subset \P^8$ and $\eta\c \P^8\dto \P^4$
be as in \Cref{prp:smooth8}.

\ref{prp:P8:a}
Let $\pi_1,\pi_2\c\P^1\times\P^1\to \P^1$ denote the projections to the
first and second component of~$\P^1\times\P^1$, \resp.
It follows from \Cref{prp:smooth8} that $X_N$ is covered by no more than two
pencils of conics with members $\{f(\pi_i^{-1}(t))\}_{t\in\P^1}$ for $i\in\{1,2\}$
and these pencils are base point free.
The set of projected complex conics~$\{(\eta\circ f)(\pi_i^{-1}(t))\}_{t\in\P^1}$ defines
for each~$i\in\{1,2\}$ a pencil of circles on~$X$.
Thus, if a pencil of circles on~$X$ has a base point,
then this complex point must
have a complex curve in~$X_N$ as preimage \wrt $\eta$.
Since $\deg X_N=\deg X$, the center of the linear projection~$\eta$ does not meet $X_N$.
We observe that $\eta=\eta_4\circ\cdots\circ\eta_7$ where
$\eta_i\c\P^{i+1}\dto\P^i$ is for all~$7\geq i\geq 4$ a linear projection
whose center lies outside the surface~$(\eta_{i+1}\circ\cdots\circ\eta_7)(X_N)$.
We claim that for all $7\geq i\geq 4$ and complex curves~$C\subset X_N$
the complex image~$(\eta_i\circ\cdots\circ\eta_7)(C)$
is not a complex point~$p$.
Indeed, this would imply that \Wlog the Zariski closure of the preimage~$\eta_i^{-1}(p)$
is a complex line in
$(\eta_{i+1}\circ\cdots\circ\eta_7)(X_N)$
that passes through the center of~$\eta_i$ although this center lies outside the surface.
Since $\eta$ does not contract complex curves to complex points,
we deduce that the pencils of circles on~$X$ must be base point free as well.

\ref{prp:P8:b}
By \Cref{prp:smooth8}, we have $L=\eta(C)$ for some complex curve~$C\subset X_N$ \st $\deg C\geq 2$.
Since $X_N$ does not contain complex lines and $\S^3$ does not contain real lines,
it follows that $L\subset\Sing X$ is a non-real line component.

\ref{prp:P8:c}
Suppose by contradiction that~$q\in \Sing X$ is a complex isolated singularity
and let~$\rho\c\P^4\dto\P^3$ be the complex linear projection with center~$q$.
Let
$H\subset X$ be a general hyperplane section of $X$,
and let $H_q\subset X$ be a general hyperplane section of $X$ containing $q$.
We observe that $H_q$ is the preimage of a general hyperplane section~$\rho(H_q)$ of~$\rho(X)$.
It follows from \Cref{lem:genus}\ref{lem:genus:a} that
$\Sing H=H\cap \Sing X$ and
$\Sing \rho(H_q)=\rho(H_q)\cap\Sing \rho(X)$.
Since the restricted map~$\rho|_{X}\c X\dto \rho(X)$ is birational,
we find that $\Sing H_q=H_q\cap \Sing X$ so that $|\Sing H_q|=|\{q\}|+|\Sing H|$.
Hence, $p_g(H_q)<p_g(H)$ by \Cref{lem:genus}\ref{lem:genus:c}.
It follows from \citep[Theorem~5]{2001} that $0\leq p_g(H)\leq 1$ and thus $p_g(H_q)=p_g(\rho(H_q))=0$.
We know from \ASN{prp:P8:a} that $q$ is not a base point for a pencil of circles.
Therefore, $\rho(X)$ is covered by two pencils of conics
and not covered by complex lines.
In other words, a general hyperplane section of $\rho(X)$ has geometric genus~0.
We arrived at a contradiction as $\rho(X)$
is by \citep[Theorem~8]{2001} either complex ruled or
contains infinitely many complex conics through a general complex point.
We conclude that the surface~$X$ does not have isolated singularities.
\end{proof}

\begin{proposition}
\label{prp:N4}
If $X\subset\P^3$ is a surface of degree four that is covered by
base point free pencils of lines and conics,
then $X$ has lattice type~$(2,1,3)$.
\end{proposition}

\begin{proof}
Suppose that $F,G\subset X\times\P^1$ denote the pencils of lines and conics.
Both pencils on the non-planar surface~$X$ are base point free,
and thus for almost all $(i,j)\in\P^1\times\P^1$
the members $F_i$ and $G_j$ intersect in a single complex point~$p_{ij}$.
Therefore, there exists a birational map $f\c\P^1\times\P^1\dto X$ that sends~$(i,j)$ to~$p_{ij}$.
Let $W$ denote the 6-dimensional vector space of bidegree~$(2,1)$ forms on~$\P^1\times \P^1$.
The parameter lines of $f$ are lines and conics,
and thus the components of~$f$ must form a basis for some 4-dimensional subspace of~$W$.
A map whose components form a basis for $W$
defines a birational morphism~$g\c\P^1\times\P^1\to X'\subset\P^5$ \st $\deg X'=4$.
Notice that $f$ is defined by the composition of
$g$ with a linear projection~$u\c\P^5\dto \P^3$.
Since $\deg X'=\deg X$, the center of $u$ lies outside~$X'$.
This implies that $f$ is a morphism and thus a smooth model for $X$.
We conclude that $N(X)=\bas{\l_0,\l_1}_\Z$ and $\h=2\,\l_0+\l_1$.
We know from \citep[Example~II.8.20.3]{1977} that $\k=-2\,\l_0-2\,\l_1$.
It follows \Cref{lem:delta}\ref{lem:delta:a} that $\delta(X)=3$ and thus
$X$ has lattice type~$(2,1,3)$.
\end{proof}

\section{Singular components via central projection}
\label{sec:central}

The central projection of a great celestial surface $X\subset\S^3$
of degree eight is a surface in $\P^3$ of degree four.
We show that the intersection
of this quartic surface with the branching locus
consist of four complex lines.
We then argue that the ramification locus $X\cap\E$ must be a union of two
left generators and two right generators.
This allows us to recover in \Cref{prp:t} the complete singular locus of $X$
by using the sectional delta invariant.
Since the right Clifford translations the
two right generators in~$X$ invariant,
we conclude in \Cref{prp:star} that $X(\R)$ is a pointwise Hamiltonian product
of circles in~$S^3$.

The \df{central projection} $\tau\c\S^3\to\P^3$
sends $(x_0:\ldots:x_4)$ to $(x_1:x_2:x_3:x_4)$.
Therefore, $\tau$ is a 2:1 linear projection
with ramification locus~$\E$ and
branching locus
\[
\tau(\E)=\set{y\in \P^3}{y_0^2+y_1^2+y_2^2+y_3^2=0}.
\]
We will call the central projection of a left generator
or right generator
into the branching locus $\tau(\E)$,
also \df{left generator} and \df{right generator}, \resp.

Notice that a fiber of the central projection $S^3\to\R^3$
induced by $\tau$
consists of antipodal points of $S^3$ and that great circles are
send to lines.

\begin{notation}
\label{ntn:X}
Let $\X\subset\S^3$ denote a great celestial surface of degree~$8$.
\END
\end{notation}

\begin{lemma}
\label{lem:tp}~
\begin{claims}
\item\label{lem:tp:a}
The surface~$\tau(\X)\subset\P^3$ is of degree~$4$
and has no isolated singularities.

\item\label{lem:tp:b}
The surface $\tau(\X)$
is covered by
exactly one pencil of lines
and
exactly one pencil of conics.
Both these pencils are base point free.

\item\label{lem:tp:c}
The lattice type of~$\tau(\X)$ is equal to~$(2,1,3)$.

\end{claims}
\end{lemma}

\begin{proof}
\ref{lem:tp:a}
The great circles in $\X$ are centrally projected to lines in $\tau(\X)$.
It follows that $\tau$ defines a 2:1 covering $\X\to \tau(\X)$,
which is a local complex analytic isomorphism on each of its two sheets.
By definition, $\deg \tau(\X)=\tau(\X)\cap H\cap H'$ for
some general hyperplanes $H,H'\subset\P^3$
and the preimages $\tau^{-1}(H),\tau^{-1}(H')\subset\P^4$
define hyperplanes that pass through $(1:0:0:0:0)$.
It follows that $\deg\tau(\X)=\tfrac{1}{2}\deg\X=4$.
We deduce from \Cref{prp:P8}\ref{prp:P8:c}
that $\Sing\X$, and thus $\Sing\tau(\X)$, does not contain isolated singularities.

\ref{lem:tp:b}
Because $\deg \tau(\X)\neq 2$, it follows that $\tau(\X)$ is not doubly ruled,
and thus $\X$ is not covered by two pencils of great circles.
\ASN{lem:tp:b} is now a
straightforward consequence of \Cref{prp:P8}\ref{prp:P8:a}.

\ref{lem:tp:c}
Direct consequence of
\Cref{prp:N4} and Assertions~\ref{lem:tp:a} and \ref{lem:tp:b}.
\end{proof}

A curve in a surface is called a \df{double curve} if a general
complex point in this curve has multiplicity two in the surface.
This is compatible with the definition of double circle in \Cref{sec:intro}.

\begin{lemma}
\label{lem:tE}
The surface $\tau(\X)$ has lattice type~$(2,1,3)$
and there exist
two pairs of complex conjugate
generators $\tau(R)$, $\tau(\oR)$ and $\tau(L)$, $\tau(\oL)$,
and a line~$\tau(V)$
\st
their common incidences are as in \Cref{fig:tE} and
\begin{Mlist}
\item
$\tau(\X)\cap\tau(\E)=\tau(L)\cup\tau(\oL)\cup\tau(R)\cup\tau(\oR)$,
\item
$\Sing\tau(\X)=\tau(R)\cup\tau(\oR)\cup\tau(V)$ consist of three complex double lines,
\item
$[\tau(L)]=[\tau(\oL)]=\l_0$,
\item
$[\tau(R)]=[\tau(\oR)]=\l_1$ and $\Delta(\tau(R))=\Delta(\tau(\oR))=1$,
\item
$[\tau(V)]\in\{2\,\l_0,~\l_0\}$ and $\Delta(\tau(V))=1$.
\end{Mlist}
\end{lemma}

\begin{figure}[!ht]
\centering
\begin{tikzpicture}[xscale=0.3,yscale=0.3] % clifford one circle
%\draw[help lines] (-4,4) grid (4,-4);
\draw[thick,red] (-6, 2) -- ( 6, 2) node[right] {$\tau(R)$};
\draw[thick,red] (-6,-2) -- ( 6,-2) node[right] {$\tau(\oR)$};;
\draw[thick,blue] (-4, 4) -- (-4,-4) node[below] {$\tau(L)$};
\draw[thick,blue] ( 4, 4) -- ( 4,-4) node[below] {$\tau(\oL)$};
\draw[thick, black!50!green] ( 0, 4) -- ( 0,-4) node[below] {$\tau(V)$};
\draw[draw=black!50!green, fill=green!5] ( 0, 2) circle [radius=0.2] ;
\draw[draw=black!50!green, fill=green!5] ( 0,-2) circle [radius=0.2] ;
\draw[draw=blue, fill=blue!5!white]      (-4, 2) circle [radius=0.2];
\draw[draw=blue, fill=blue!5!white]      ( 4, 2) circle [radius=0.2];
\draw[draw=blue, fill=blue!5!white]      (-4,-2) circle [radius=0.2];
\draw[draw=blue, fill=blue!5!white]      ( 4,-2) circle [radius=0.2];
\end{tikzpicture}
\caption{Each line segment represents a complex line in~$\tau(\X)$.
Two line segments meet at a disk if and only if their corresponding complex lines intersect.}
\label{fig:tE}
\end{figure}

\begin{proof}
Recall from \Cref{lem:tp}\ref{lem:tp:c} that $\tau(\X)$ has lattice type~$(2,1,3)$.
Let $E$ be the scheme-theoretic intersection of $\tau(\X)$
with the hyperquadric~$\tau(\E)$ so that
\[
[E]=2\,\h=4\,\l_0+2\,\l_1.
\]
With ``scheme-theoretic'',
we mean that we take the ideal of the intersection into account and not just its zero set.
Let $A$ be a general great circle and let $B$ be a general small circle in $\X$.
Their central projections are a line~$\tau(A)$ and an irreducible conic~$\tau(B)$ in~$\tau(\X)$
(see \Cref{lem:tp}\ref{lem:tp:b}).
By \Cref{rmk:lt}, we have
\[
[\tau(A)]=\l_0
\quad\text{and}\quad
[\tau(B)]=\l_1.
\]
Since $\E_\R=\varnothing$,
both $A$ and $B$ meet the hyperplane section~$\E\subset\S^3$ in complex conjugate points.
By B\'ezout's theorem, $|\E\cap A|=|\E\cap B|=2$ and thus
\[
|E\cap\tau(A)|=2
\quad\text{and}\quad
|E\cap\tau(B)|=2.
\]
On the other hand, the scheme theoretic intersection numbers
are
\[
[E]\cdot [\tau(A)]=2
\quad\text{and}\quad
[E]\cdot [\tau(B)]=4.
\]
The line~$\tau(A)$ intersects~$\tau(\E)$ transversally in two complex conjugate points,
and thus $\tau(A)$ intersects $E$ transversally as well.
It follows that there exists possibly reducible components $F$ and $F'$
\st
\begin{Mlist}
\item $E=F\cup F'$, $|F\cap\tau(B)|=2$, $|F'\cap\tau(A)|=2$,
\item $[E]=2\,[F]+[F']$, $[F]\cdot [\tau(B)]=2$ and $[F']\cdot [\tau(A)]=2$.
\end{Mlist}
Hence, it follows that
\[
[F]=2\,\l_0
\quad\text{and}\quad
[F']=2\,\l_1.
\]
We know from \Cref{lem:gg}\ref{lem:gg:b} that $\varphi^{-1}(F)$ is reducible,
where $\varphi\c\P^1\times\P^1\to\tau(\X)$ denotes a smooth model.
Since $F\subset\tau(\E)$ is real, it follows that
\[
F=\tau(L)\cup \tau(\oL)
\quad\text{and}\quad
[\tau(L)]=[\tau(\oL)]=\l_0.
\]
Similarly, the preimage~$\varphi^{-1}(F')$ is by \Cref{lem:gg}\ref{lem:gg:b} reducible
and thus
there exists irreducible complex conjugate curves $\tau(R)$ and $\tau(\oR)$
\st
\[
F'=\tau(R)\cup \tau(\oR)
\quad\text{and}\quad
[\tau(R)]=[\tau(\oR)]=\l_1.
\]
Since $|\tau(A)\cap F'|=2$ and $\tau(\X)$ is not a plane spanned by $\tau(F')$,
we find that $F'$ is not an irreducible double conic.
Hence, $\tau(R)$ and $\tau(\oR)$ are either
\begin{Mlist}
\item complex conjugate irreducible conics, or
\item complex conjugate double lines.
\end{Mlist}
We know from \Cref{prp:N2} that $Q:=\tau(\E)$ has lattice type~$(1,1,0)$.
Let $\h_Q$ denote its class of hyperplane sections, $\k_Q$ its canonical class
and $[C]_Q$ the class of a complex curve~$C\subset Q$.
We have $[E]_Q=4\,h_Q=4\,\l_0+4\,\l_1$,
since $\deg\tau(\X)=4$ and $E$ is the scheme-theoretic intersection
of $Q$ with $\tau(\X)$.
We may assume \Wlog that $[\tau(L)]_Q=[\tau(\oL)]_Q=\l_0$ so that $[F]_Q=2\,\l_0$.
As $[F]\cdot [F']=4$, we find that $|F\cap F'|\leq 4$ and thus $[F]_Q\cdot [F']_Q\leq 4$.
We deduce that $[E]_Q=2\,[F]_Q+2\,[F']_Q$ with $[F]_Q=2\,\l_0$ and $[F']_Q=2\,\l_1$.
Hence, $[\tau(R)]_Q=[\tau(\oR)]_Q=\l_1$, which implies that
$\tau(R)$ and $\tau(\oR)$ are complex lines, and thus complex double lines in~$\tau(\X)$.

Since $\delta(\tau(\X))=3$,
it follows from Axioms~\ref{a1} and~\ref{a3} at~\Cref{def:delta} that
\[
\Delta(\tau(R))=\Delta(\tau(\oR))=1.
\]
Hence, by Axioms~\ref{a1}, \ref{a2} and~\Cref{lem:tp}\ref{lem:tp:a}, the remaining component of~$\Sing\X$ consists of
a double line~$\tau(V)$ \st
\[
\Delta(\tau(V))=1.
\]
If $H\subset\tau(X)$ is a hyperplane section containing the double line~$\tau(V)$,
then there exists $\alpha,\beta\in\Z_{>0}$ \st
\[
[H]=\alpha\,[\tau(V)]+\beta\,[H\setminus\tau(V)]=\h=2\,\l_0+\l_1.
\]
Therefore, we deduce that
\[
[\tau(V)]\in\{2\,\l_0,~\l_0,~\l_1,~\l_0+\l_1\}.
\]
First, suppose by contradiction that $[\tau(V)]=\l_1$.
We observe that $[\tau(V)\cup \tau(R)]=\alpha\,\l_1\neq\h$
for any $\alpha\in\Z_{>0}$.
Hence, it follows from B\'ezout's theorem that $\tau(V)\cap \tau(R)$
is not a hyperplane section of~$\tau(X)$.
This implies that $\tau(V)\cap \tau(R)=\varnothing$.
Moreover, $\tau(R)$, $\tau(\oR)$ and~$\tau(V)$ form three skew double lines
each with class~$\l_1$.
Recall that each line in $\tau(\X)$ has class~$\l_0$
and is a member of the pencil of lines that covers~$\tau(\X)$.
Each line in this pencil meets the three skew lines with class~$\l_1$.
We arrived at a contradiction, since $\tau(\X)$ is a not a doubly ruled quadric.

Next, suppose by contradiction that $[\tau(V)]=\l_0+\l_1$.
The double line $\tau(V)$ meets in this case the double line $\tau(R)$, because $[\tau(R)]=\l_1$.
Thus $\tau(V)\cup \tau(R)$ forms a hyperplane section by B\'ezout's theorem
so that $[\tau(V)\cup \tau(R)]=\h$.
This is a contradiction, since
$[\tau(V)\cup \tau(R)]=\alpha\,(\l_0+\l_1)+\beta\,\l_1\neq \h$ for all $\alpha,\beta\in\Z_{>0}$.

We established that
\[
[\tau(V)]\in\{2\,\l_0,~\l_0\}.
\]
Suppose by contradiction $\tau(V)\cap \tau(L)\neq \varnothing$.
In this case there exists by B\'ezout's theorem some line~$\tau(L')\subset\tau(\X)$ \st
$\tau(V)\cup \tau(L)\cup \tau(L')\subset\tau(\X)$ is a hyperplane section
and $[\tau(L)]=[\tau(L')]=\l_0$.
Thus, there exists $\alpha,\beta,\gamma\in\Z_{>0}$ \st
\[
[\tau(V)\cup \tau(L)\cup \tau(L')]=\alpha\,[\tau(V)]+\beta\,[\tau(L)]+\gamma\,[\tau(L')]=\h.
\]
We arrived at a contradiction,
as we already established that $[\tau(V)]\neq\alpha\,\l_1$.
Since
\[
\tau(V)\cap \tau(L)=\tau(V)\cap \tau(\oL)=\varnothing,
\]
the incidences as specified in \Cref{fig:tE} are correct.
\end{proof}

\begin{proposition}
\label{prp:t}
The surface $\X$ has lattice type~$(2,2,8)$
and there exist two pairs of complex conjugate generators~$R$, $\oR$
and $L$, $\oL$, and a great circle~$V$
\st their common incidences are as in \Cref{fig:t} and
\begin{Mlist}
\item $\X\cap\E=L\cup\oL\cup R\cup\oR$,
\item $\Sing\X=L\cup\oL\cup R\cup\oR\cup V$ consists of complex double curves,
\item $[L]=[\oL]=\l_0$, $\Delta(L)=\Delta(\oL)=1$,
\item $[R]=[\oR]=\l_1$, $\Delta(R)=\Delta(\oR)=2$,
\item $[V]\in\{2\,\l_0,~\l_0\}$ and $\Delta(V)=2$.
\end{Mlist}
Moreover, if $H\subset \P^4$ is a hyperplane containing the great double circle~$V$,
then there exists small circles $C$ and $C'$ \st
\[
H\cap \X=V\cup C\cup C'
\]
and $C(\R)$, $C'(\R)$ are antipodal small circles
in the Euclidean 2-sphere~$(H\cap\S^3)(\R)$.
General great and small circles in $\X$ have classes~$\l_0$ and $\l_1$, \resp.
\end{proposition}

\begin{figure}[!ht]
\centering
\begin{tikzpicture}[xscale=0.3,yscale=0.3] % clifford one circle
%\draw[help lines] (-4,4) grid (4,-4);
\draw[thick, red]  (-6, 2) -- ( 6, 2) node[right] {$R$};
\draw[thick, red]  (-6,-2) -- ( 6,-2) node[right] {$\oR$};;
\draw[thick, blue] (-4, 4) -- (-4,-4) node[below] {$L$};
\draw[thick, blue] ( 4, 4) -- ( 4,-4) node[below] {$\oL$};
\draw[thick, black!50!green] (0,2) to [out=180, in=180] (0,-2) to [out=0, in=0] (0,2);
\draw[black!50!green] (1.9,0) node {$V$};
\draw[draw=black!50!green, fill=green!5] ( 0, 2) circle [radius=0.2] ;
\draw[draw=black!50!green, fill=green!5] ( 0,-2) circle [radius=0.2] ;
\draw[draw=blue, fill=blue!5!white]      (-4, 2) circle [radius=0.2];
\draw[draw=blue, fill=blue!5!white]      ( 4, 2) circle [radius=0.2];
\draw[draw=blue, fill=blue!5!white]      (-4,-2) circle [radius=0.2];
\draw[draw=blue, fill=blue!5!white]      ( 4,-2) circle [radius=0.2];
\end{tikzpicture}
\caption{Each line segment represents a complex double line
and the green loop represents a great double circle in~$\X$.
The line segments and/or loop meet at a disk if and only if their corresponding components in~$\Sing\tau(X)$ intersect.}
\label{fig:t}
\end{figure}

\begin{proof}
It follows from \Cref{lem:tE} and \Cref{prp:P8}\ref{prp:P8:b} that
\[
\X\cap\E=L\cup\oL\cup R\cup\oR
\quad\text{and}\quad
\Sing\X =L\cup\oL\cup R\cup\oR\cup V.
\]
By \Cref{lem:tE} the incidences are as in \Cref{fig:t}.
The hyperplane section~$\X\cap\E$ is scheme theoretically of degree 8 by B\'ezout's theorem,
and thus $L$, $\oL$ and $R$, $\oR$ are pairs of complex conjugate double lines.
The central projection~$\tau(V)$ is a double line by \Cref{lem:tE},
which implies that $V$ is a great double circle.

The surface~$\X$ has lattice type~$(2,2,8)$ by \Cref{prp:N8}.
By \Cref{lem:tp}\ref{lem:tp:b},
we may assume \Wlog that $\l_0$ and $\l_1$
are the classes of a great circle and small circle, \resp.
By comparing with \Cref{lem:tE}, we deduce that
\[
[L]=[\oL]=\l_0
\quad\text{and}\quad
[R]=[\oR]=\l_1.
\]
The planar sections of the quartic surface~$\tau(\X)\subset\P^3$ that contain
the double line~$\tau(V)$ define a pencil of conics.
The preimage~$H\subset\S^3$ of a plane containing~$\tau(V)$ is an Euclidean 2-sphere that contains
the great circle~$V$ and two small circles $C$ and $C'$ that are centrally projected 2:1 to a conic in $\tau(\X)$.
We have $[H\cap\X]=\alpha\,[V]+2\,\l_1=2\,\l_0+2\,\l_1$ for some $\alpha\in\Z_{>0}$ and thus
\[
[V]\in\{2\,\l_0,~\l_0\}.
\]
Since $\delta(\X)=8$,
it follows from \AXM{a1} at \Cref{def:delta} that
\[
\Delta(R)+\Delta(\oR)+\Delta(L)+\Delta(\oL)+\Delta(V)=8.
\]
We observe that $\Delta(L)=\Delta(\oL)\geq 1$ and $\Delta(R)=\Delta(\oR)\geq 1$
by Axioms~\ref{a2} and~\ref{a3}.
We apply \AXM{a5} with $\rho=\tau$
and find that $\Delta(\tau(V))\cdot 2\geq \Delta(V)\cdot 1$, where $\Delta(\tau(V))=1$ by \Cref{lem:tE}.
As $\Delta(V)\geq \deg V=2$ by \AXM{a2}, we deduce that
\[
\Delta(V)=2.
\]
Since $\tau(L),\tau(\oL)\nsubseteq\Sing\tau(\X)$ and $\Delta(\tau(R))=\Delta(\tau(\oR))=1$,
we have
\[
\Delta(L)=\Delta(\oL)=1
\quad\text{and}\quad
\Delta(R)=\Delta(\oR)=2.
\]
We concluded the proof.
\end{proof}

\begin{remark}
\label{rmk:LR}
Notice that in \Cref{prp:t}, either $L,\oL\subset \E$ are both left generators
and $R,\oR\subset \E$ are both right generators, or vice versa.
\END
\end{remark}

\begin{lemma}
\label{lem:star}
If $A\subset\X$ is a great circle and $B\subset\X$
is a small circle \st their intersection~$A(\R)\cap B(\R)$ contains the identity quaternion~$\mathbf{1}\in S^3$,
then
$\X(\R)$ is equal to either $A(\R)\star B(\R)$ or $B(\R)\star A(\R)$.
\end{lemma}

\begin{proof}
We know from \Cref{prp:t} that the great circles in~$\X$ meet the elliptic absolute~$\E$
at the double generators $R$ and $\oR$.
Recall from \Cref{rmk:LR} that these generators are either both left or both right.
Let $\varepsilon\in\S^3_\R$ \st $\bR(\varepsilon)=\mathbf{1}$.

First, we suppose that $R$ and $\oR$ are both right generators.

Let $\cZ\subset \S^3$ be defined as the Zariski closure of
$\set{\alpha\,\hstar\,\beta}{\alpha\in A\setminus\E,~\beta\in B\setminus\E }$
and let $F\subset \cZ\times B$ be the pencil of circles on the surface~$\cZ$
\st $F_\beta=\varphi_\beta(A)$ for all $\beta\in B\setminus\E$,
where $\varphi_\beta\in\rt\S^3$ sends $x$ to $x\,\hstar\,\beta$.
It follows from \Cref{prp:E}\ref{prp:E:a} that $F_\beta(\R)=\set{a\,\star\,\bR(\beta)}{a\in A(\R)}$
for all $\beta\in B_\R$ and thus $\cZ(\R)=A(\R)\star B(\R)$.
We know from \Cref{prp:E}\ref{prp:E:b} that $\rt\S^3\subset\aut_\E\S^3$ and thus
infinitely many members of~$F$ are great circles.

Since $\varepsilon\in B$ and $F_\varepsilon=A$,
we find that $A\subset\X\cap\cZ$, $A\cap R\neq\varnothing$ and $A\cap\oR\neq\varnothing$.
We now apply \Cref{prp:E}\ref{prp:E:d} and deduce
that $F_\beta\cap R\neq\varnothing$ and $F_\beta\cap\oR\neq \varnothing$ for all~$\beta\in B$.

Suppose by contradiction that $\cZ\neq\X$.
In this case there exists $\beta\in B_\R$ \st $F_\beta$ is a great circle that is not contained in~$\X$.
We observe that $\beta\in F_\beta$, since $\varepsilon\in A$ by assumption.
Let $C$ be a great circle in~$\X$ \st $\beta\in C$.
The incidence relations for the current scenario are schematically depicted in \Cref{fig:converse}.
\begin{figure}[!ht]
\centering
\begin{tikzpicture}[yscale=0.8]
%\draw[help lines] (-4,4) grid (4,-4);
%
\draw[thick, red] (-2,1) -- (2,1) node[right] {$R$};
\draw[thick, red] (-2,-1) -- (2,-1) node[right] {$\oR$};
\draw[thick, blue] (-1,2) to [out=260, in=100] (-1,-2) node[below] {$A$};
\draw[thick, blue] (1,2) to [out=260, in=100] (1,-2) node[below] {$F_\beta$};
\draw[thick, densely dotted, blue] (0,2) to [out=290, in=125] (2,-2) node[below right] {$C$};
\draw[thick, black!20!green] (-3,-0.2) to [out=10, in=170] (3,-0.2) node[right] {$B$};
\draw[draw=black, fill=white] (0.8,0.05) circle  [radius=0.08] node[black, above right] {$\beta$};
\draw[draw=black, fill=white] (-1.2,0.05) circle [radius=0.08] node[black, above left] {$\varepsilon$};
\draw[draw=black, fill=white] (-1.15,1) circle   [radius=0.08];
\draw[draw=black, fill=white] (-1.15,-1) circle  [radius=0.08];
\draw[draw=black, fill=white] (1-0.15,1) circle  [radius=0.08];
\draw[draw=black, fill=white] (1-0.15,-1) circle [radius=0.08];
\draw[draw=black, fill=white] (0.35,1) circle    [radius=0.08];
\draw[draw=black, fill=white] (1.35,-1) circle   [radius=0.08];
\end{tikzpicture}
\caption{See the proof of \Cref{prp:star}. Each curve segments correspond to
one of the complex curves~$A,B,C,R,\oR\subset \X$ or $F_\beta\subset\cZ$, where $\beta\in B_\R$.
Two curve segments meet at a disk if and only if the corresponding complex curves intersect.
The incidence point~$\varepsilon\in\S^3_\R$ corresponds via $\bR$ to the identity quaternion $\mathbf{1}\in S^3$.
}
\label{fig:converse}
\end{figure}

The central projections $\tau(F_\beta)$ and $\tau(C)$
are lines that meet the right generators $\tau(R)$ and $\tau(\oR)$ in
the complex doubly ruled quadric~$\tau(\E)$ \st $\tau(F_\beta)\cap \tau(C)=\{\tau(\beta)\}$.
We arrived at a contradiction as $\tau(F_\beta)$ and $\tau(C)$ span a plane
so that $\tau(R)$ and $\tau(\oR)$ cannot be skew.
Hence, $\cZ=\X$ so that $\X(\R)=A(\R)\star B(\R)$.

Finally, we suppose that $R$ and $\oR$ are both left generators.
In this case, we define $F\subset \cZ\times B$ to be the pencil of circles \st
$F_\beta=\varphi_\beta(A)$ for all $\beta\in B\setminus\E$,
where $\varphi_\beta\in\lt\S^3$ sends $x$ to $\beta\,\hstar\,x$.
The analoguous proof as before shows that $\X(\R)=B(\R)\star A(\R)$.
\end{proof}

\begin{proposition}
\label{prp:star}
There exists a great circle $A\subset\S^3$ and small circle $B\subset\S^3$
\st $\X(\R)$ is equal to either $A(\R)\star B(\R)$ or $B(\R)\star A(\R)$.
\end{proposition}

\begin{proof}
Let $\varepsilon\in\S^3_\R$ \st $\bR(\varepsilon)=\mathbf{1}$
is the identity quaternion~$\mathbf{1}\in S^3$.
Notice that the left and right Clifford translations send great circles in $\S^3$ to great circles
as they preserve the point~$(1:0:0:0:0)$.

By \Cref{prp:E} there exists
a right Clifford translation~$\varphi\in \rt\S^3\cap\aut\S^3$ \st $\varepsilon\in \varphi(\X)$.
Similarly, there exists a left Clifford translation~$\varphi\in \lt\S^3\cap\aut\S^3$
\st $\varepsilon\in \varphi(\X)$.
We now apply \Cref{lem:star} and find that
there exists a great circle $A\subset\S^3$, a small circle $C\subset\S^3$
and $q\in \S^3_\R$
\st either
\begin{Mlist}
\item $\varphi(\X)(\R)=A(\R)\star C(\R)$ for some $\varphi\in \rt\S^3$ that sends $x$ to $x\,\hstar\,q$, or
\item $\varphi(\X)(\R)=C(\R)\star A(\R)$ for some $\varphi\in \lt\S^3$ that sends $x$ to $q\,\hstar\,x$.
\end{Mlist}
In the first case, we have $\X(\R)=A(\R)\star B(\R)$, where the small circle $B$
is defined as the Zariski closure of $\set{c\,\hstar\,q}{c\in C\setminus\E}$.
In the second case, we have
$\X(\R)=B(\R)\star A(\R)$, where the small circle $B$
is defined as the Zariski closure of $\set{q\,\hstar\,c}{c\in C\setminus\E}$.
This concludes the proof.
\end{proof}

\section{Shapes}
\label{sec:shape}

In this section, we investigate the shapes of great celestial surfaces of degree eight
and prove \Cref{thm:shape} and its corollaries.
We use the fact that such a surface is the image of a continuous map $S^1\times S^1\to S^3$
whose parameter lines are circles.

Let $\pi_1,\pi_2\c \P^1\times\P^1\to\P^1$ denote the projections
to the first and second component of $\P^1\times \P^1$, \resp.
A fiber of $\pi_1$ and $\pi_2$ is called a \df{left fiber} and \df{right fiber}, \resp.
The image of a left fiber or right fiber
with respect a given morphism $\P^1\times\P^1\to\S^3$
is called a \df{left image} and \df{right image}, \resp.
A left/right fiber/image is real unless explicitly stated otherwise.

Suppose that $X\subset\S^3$ is a surface and let $V$ be the Zariski closure of~$\Sing X_\R$.
We call $\varphi\c \P^1\times\P^1\to X$ a \df{great morphism}
if it is a birational morphism \st its left and right images are circles
and $V$ is a circle whose preimage~$\varphi^{-1}(V)$
consists of either
\begin{Mlist}
\item two left fibers,
\item a single left fiber, or
\item two complex conjugate left fibers.
\end{Mlist}
We shall refer to $V\subset\Sing X$ as the \df{special left image} of~$\varphi$.

\begin{lemma}
\label{lem:gm}
If $X\subset \S^3$ is a great celestial surface of degree eight,
then there exists a great morphism~$\varphi\c \P^1\times\P^1\to X$
\st the left images and right images correspond to the great and small
circles, \resp.
\end{lemma}

\begin{proof}
By \Cref{prp:smooth8} there exists a
smooth model~$\varphi\c \P^1\times \P^1\to X$ whose components are forms on $\P^1\times\P^1$
of bidegree~$(2,2)$.
It follows that the complex left images and complex right images
are complex conics and/or complex double lines in~$X$ for all complex~$p\in\P^1$.
We know from \Cref{prp:P8}\ref{prp:P8:b} that the complex double lines are non-real.
This implies that the left images and right images are circles.
The Zariski closure~$V$ of $\Sing X_\R$ is by \Cref{prp:t}
is a great circle \st $[V]\in\{\l_0,2\,\l_0\}$.
If $[V]=2\,\l_0$, then
it follows from \Cref{lem:gg}\ref{lem:gg:b} that $\varphi^{-1}(V)$
consists of two complex left fibers. These left fibers are either both real or
complex conjugate.
If $[V]=\l_0$, then $\varphi^{-1}(V)$ consists of a single left fiber.
We established that $\varphi$ is a great morphism with special left image~$V$.
By \Cref{prp:t}, we may assume \Wlog that the left and right images are great and small circles, \resp.
\end{proof}

Suppose that $\varphi\c \P^1\times\P^1\to X$ is a great morphism
with special left image~$V$.
Let $Z:=X(\R)$, $C:=V(\R)$, $L_p:=\{p\}\times S^1$ and $R_p:=S^1\times\{p\}$ for all $p\in S^1$.
Since $\P^1_\R\times\P^1_\R\cong S^1\times S^1$ and $\S^3_\R\cong S^3$,
the great morphism~$\varphi$ induces a birational morphism $\xi\c S^1\times S^1\to Z$.
Moreover, there exist $p,q\in S^1$ \st either
\begin{Mlist}
\item $\xi^{-1}(C)=L_p\cup L_q$,
\item $\xi^{-1}(C)=L_p$, or
\item $\xi(S^1\times S^1)=Z\setminus C$.
\end{Mlist}
We call the morphism~$\xi$ induced by $\varphi$ an \df{induced great morphism}.
We call $C$ the \df{double left circle}.
For all $p\in S^1$, we refer to $L_p$, $R_p$, $\xi(L_p)$ and $\xi(R_p)$
as a \df{left fiber}, \df{right fiber}, \df{left circle} and \df{right circle}, \resp.

\begin{lemma}
\label{lem:c123}
If $\xi\c S^1\times S^1\to Z$
is an induced great morphism with double left circle~$C$,
then there exist tori $T, T'\subset S^3$
\st for all left circles~$A\subset Z$ and right circles~$B\subset Z$
one of the following three cases holds:
\begin{Menum}{(\roman*)}{(\roman*)}
\item\label{c1}
$\xi^{-1}(C)$ consists of two left fibers,
$Z=T\cup T'$, $C=T\cap T'$ \st $C$ is non-trivial in both $T$ and $T'$,
$|B\cap C|=2$ and either $A\subset T$ or $A\subset T'$.
\item\label{c2}
$\xi^{-1}(C)$ consists of a single left fiber,
$Z=T$ and $|B\cap C|=1$.
\item\label{c3}
$\xi(S^1\times S^1)=Z\setminus C$,
$Z=T\cup C$, $T\cap C=\varnothing$ and $|B\cap C|=0$.
\end{Menum}
\end{lemma}

\begin{proof}
By definition, there exists $p,q\in S^1$ \st
either $\xi^{-1}(C)\in\{L_p\cup L_q,\,L_p\}$ or $\xi(S^1\times S^1)=Z\setminus C$.
We illustrated in \Cref{fig:xi} the left fibers that map to $C$.
\begin{figure}[!ht]
\centering
\csep{4mm}
\begin{tabular}{ccc}
\begin{tikzpicture}[xscale=0.8,yscale=1]
\draw[blue] (-2,1) to (2,1);  \draw[blue,->] (-2,1) to (0,1);
\draw[blue] (-2,-1) to (2,-1);\draw[blue,->] (-2,-1) to (0,-1);
\draw[red] (-2,1) to (-2,-1);\draw[red,->] (-2,1) to (-2,0);
\draw[red] ( 2,1) to ( 2,-1);\draw[red,->] ( 2,1) to ( 2,0);
\draw[black!45!green,dashed,very thick] (-2,1) to node[right] {$L_p$} (-2,-1);
\draw[black!45!green,dashed,very thick] (2,1) to node[left] {} (2,-1);
\draw[black!45!green,dashed,very thick] (-0.3,1) to node[right] {$L_q$} (-0.3,-1);
\node at (-1,0.3) {$U$};
\node at (1,0.3) {$U'$};
\end{tikzpicture}
&
\begin{tikzpicture}[xscale=0.8,yscale=1]
\draw[blue] (-2,1) to (2,1);  \draw[blue,->] (-2,1) to (0,1);
\draw[blue] (-2,-1) to (2,-1);\draw[blue,->] (-2,-1) to (0,-1);
\draw[red] (-2,1) to (-2,-1);\draw[red,->] (-2,1) to (-2,0);
\draw[red] ( 2,1) to ( 2,-1);\draw[red,->] ( 2,1) to ( 2,0);
\draw[black!45!green,dashed,very thick] (-2,1) to node[right] {$L_p$} (-2,-1);
\draw[black!45!green,dashed,very thick] (2,1) to node[left] {} (2,-1);
%\draw[black!45!green,dashed,very thick] (-0.3,1) to (-0.3,-1);
\end{tikzpicture}
&
\begin{tikzpicture}[xscale=0.8,yscale=1]
\draw[blue] (-2,1) to (2,1);  \draw[blue,->] (-2,1) to (0,1);
\draw[blue] (-2,-1) to (2,-1);\draw[blue,->] (-2,-1) to (0,-1);
\draw[red] (-2,1) to (-2,-1);\draw[red,->] (-2,1) to (-2,0);
\draw[red] ( 2,1) to ( 2,-1);\draw[red,->] ( 2,1) to ( 2,0);
\end{tikzpicture}
\\
\ref{c1}&\ref{c2}&\ref{c3}
\end{tabular}
\caption{A torus $S^1\times S^1$ is identified with a square,
where the blue horizontal sides are identified
and red vertical sides are identified.
Any vertical or horizontal line segment in the square represents a left and right fiber, \resp.
The left fibers that are send via $\xi$ to
the double left circle~$C$ are represented by the vertical green dashed line segments.}
\label{fig:xi}
\end{figure}

First, we suppose that $\xi^{-1}(C)=L_p\cup L_q$.
In this case, $\xi$ induces a biregular isomorphism
$S^1\times S^1\setminus(L_p\cup L_q)\cong Z\setminus C$.
Hence, there exist~$U,U'\subset S^1\times S^1$ \st $S^1\times S^1=U\cup U'$ and $U\cap U'=L_p\cup L_q$
(see \Cref{fig:xi}).
Since $\xi(L_p)=\xi(L_q)=C$, we find that $\xi(U)\cong \xi(U')\cong S^1\times S^1$
so that $Z$ is the union of two tori that intersect at the double left circle~$C$.
Moreover, any left circle is contained in either the torus~$\xi(U)$ or the torus~$\xi(U')$.
A right fiber meets $L_p\cup L_q$ in two points and thus a right circle meets
the double left circle $C$ in two points a well.
We observe that $C$ is non-trivial in both~$\xi(U)$ and~$\xi(U')$.
We established that \ref{c1} holds.

Next, we suppose that $\xi^{-1}(C)=L_p$. In this case, $\xi$ induces a biregular isomorphism $S^1\times S^1\cong Z$
and the right fibers meet the left fiber~$L_p$ in a single point. It follows that \ref{c2} holds.

For the remaining case $\xi(S^1\times S^1)=Z\setminus C$, we observe that
$\xi$ induces a biregular isomorphism~$S^1\times S^1\cong Z\setminus C$ so that \ref{c3} holds.

We conclude the proof as each of the three cases is accounted for.
\end{proof}

\begin{example}
\label{exm:shapeI}
Suppose that $\xi\c S^1\times S^1\to Z$ is an induced great morphism with double left circle~$C$
\st the preimage~$\xi^{-1}(C)$ consist of two left fibers.
By \Cref{lem:gm}, the surface $Z\subset S^3$ is great celestial and of degree eight.
Moreover, the left and right circles of $\xi$ are great and small circles, \resp.
By \Cref{lem:c123} and the definition of~$\xi$, there exist tori~$T,T'\subset S^3$ \st
$Z=T\cup T'$, $C=T\cap T'$ and $C$ is non-trivial in both $T$ and~$T'$.
Moreover, each tori $T''\in\{T,T'\}$ is a disjoint union of great circles, and
each small circle in $Z$ meets a great circle~$A\subset T''\setminus C$ transversally in one point.
The small circles meet the great double circle~$C$ in two points.
See \Cref{fig:gs} for an illustration of such a surface.
\END
\end{example}

\begin{figure}[!ht]
\centering
\csep{5mm}
\begin{tabular}{cc}
{\it great circles} & {\it small circles}
\\
\fig{4}{5}{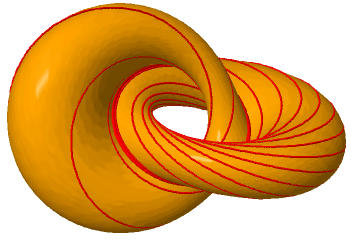} &
\fig{4}{5}{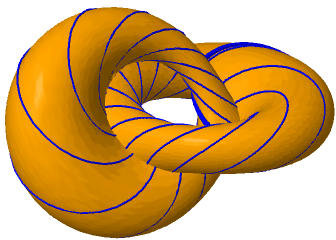}
\end{tabular}
\caption{The surface $Z$ is the union of two tori that
intersect at the double great circle~$C$.
Each of these two tori are a disjoint union of great circles
and the small circles meet $C$ in two real points.}
\label{fig:gs}
\end{figure}

\begin{lemma}
\label{lem:S}
Suppose that $Z\subset S^3$ is a great celestial surface of degree eight
and $\xi\c S^1\times S^1\to Z$ an induced great morphism with double left circle~$C$.
Then there exists $n\in \{0,1,2\}$ \st for all Euclidean 2-spheres $S\subset S^3$
containing~$C$,
\[
Z\cap S=C\cup B\cup B',
\]
where $B, B'$ are disjoint right circles \st $|B\cap C|=|B'\cap C|=n$.
\end{lemma}

\begin{proof}
Direct consequence of \Cref{prp:t,lem:c123}.
\end{proof}

\begin{example}
\label{exm:S}
Let $Z,S\subset S^3$ and $C,B, B'\subset Z$
be defined as in \Cref{lem:S}.
Suppose that $p_\infty\in C$ is the center of stereographic projection~$\mu\c S^3\dto \R^3$.
Notice that $\mu$ defines biregular isomorphism $S^3\setminus\{p_\infty\}\cong \R^3$
that sends circles containing~$p_\infty$ to lines, and Euclidean 2-spheres containing $p_\infty$
to planes.
Hence, $\mu(S)$ is a plane containing
the double line $\mu(C)$ and the circles $\mu(B)$ and $\mu(B')$.
In \Cref{fig:S}, we show three examples for $\mu(Z)$
together with the plane section $\mu(Z\cap S)$.
\begin{Mlist}
\item If $|B\cap C|=2$, then $\mu(B)$ meets $\mu(C)$ transversally in two points.
\item If $|B\cap C|=1$, then $\mu(B)$ meets $\mu(C)$ tangentially.
\item If $|B\cap C|=0$, then the torus $\mu(Z)$ and double line $\mu(C)$ are \df{linked} in~$\R^3$,
namely the circle $C$ is non-trivial in the complement $S^3\setminus Z$.
\end{Mlist}
\begin{figure}[!ht]
\centering
\csep{4mm}
\begin{tabular}{ccc}
\fig{3.5}{3.5}{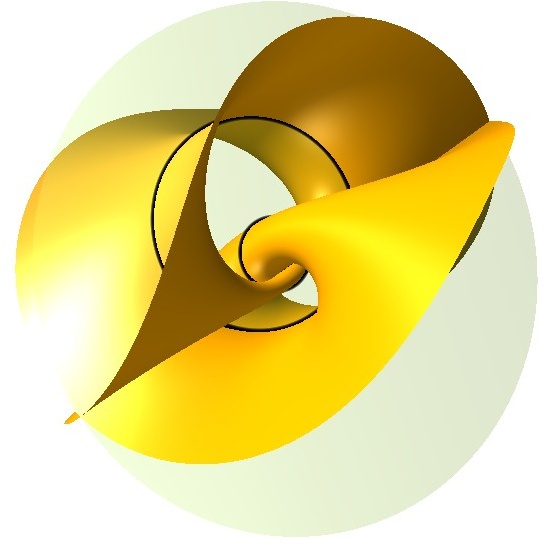}      &
\fig{3.5}{3.5}{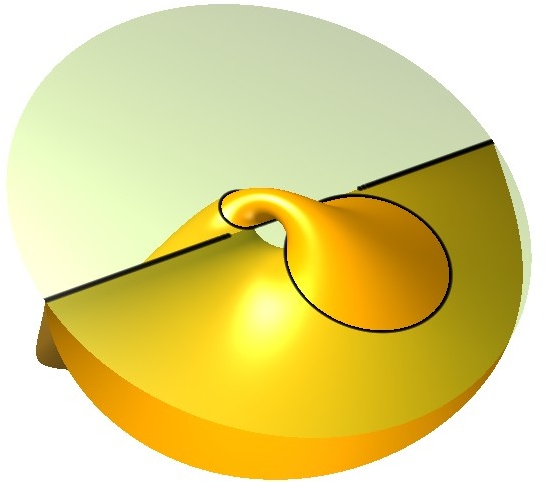}   &
\fig{3.5}{3.5}{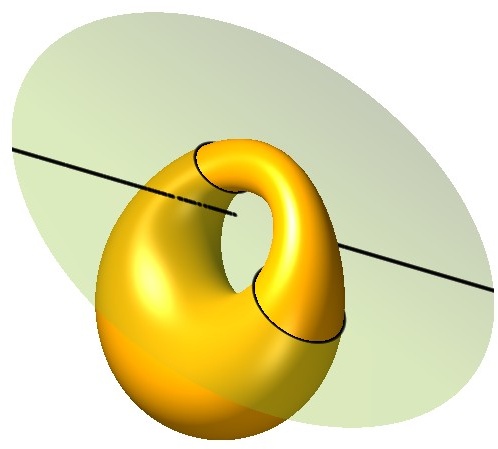}
\\
$|B\cap C|=2$ & $|B\cap C|=1$ & $|B\cap C|=0$
\end{tabular}
\caption{Stereographic projection of a Euclidean 2-sphere~$S\subset S^3$ containing
the great double circle $C$ so that $Z\cap S=C\cup B\cup B'$ for some right circles $B$ and $B'$.}
\label{fig:S}
\end{figure}

\Cref{lem:S} and \Cref{fig:S} suggests that
great celestial surfaces of degree eight in~$S^3$ could be considered as counterparts in elliptic geometry
of surfaces in~$\R^3$ that are obtained by revolving
a circle about an axis that is coplanar with the circle, namely the
spindle torus, horn torus and ring torus.
\END
\end{example}

\begin{lemma}
\label{lem:ind}
Suppose that $Z\subset S^3$ is a great celestial surface of degree eight and let
$\xi\c S^1\times S^1\to Z$ be an induced great morphism with double left circle~$C$.
\begin{claims}
\item\label{lem:ind:a}
The surface $Z$ satisfies the Conditions
\ref{C1}, \ref{C2}, \ref{C3} and \ref{C4} at \Cref{def:shape}.

\item\label{lem:ind:b}
If $\xi^{-1}(C)$ consists of a single left fiber, then the small circles in~$Z$
meet the great circle~$C$ tangentially in one point.

% \item\label{lem:ind:c}
% If $\xi(S^1\times S^1)=Z\setminus C$, then the circle~$C$ and the torus~$Z$
% are linked in~$S^3$.
\end{claims}
\end{lemma}

\begin{proof}
We know from \Cref{prp:P8}\ref{prp:P8:a} that $Z$ is 2-circled.
We recall from \Cref{lem:gm}
that the left circles are great and the right circles are small.
It follows from the definition of induced great morphisms
that $Z$ is a disjoint union of left circles and a union of right circles.
The singular locus of $Z$ is the double great circle~$C$.
We established that Assertion~\ref{lem:ind:a} holds.
Assertion~\ref{lem:ind:b} is a straightforward consequences of \Cref{lem:S}
(see also \Cref{exm:S}).
\end{proof}

\begin{proof}[Proof of \Cref{thm:shape}]
\ASN{thm:shape:a}
follows from
\citep[Theorem~1]{2021circle} and
\ASN{thm:shape:b}
follows from
\citep[Theorem~1(c) and Corollary~4]{2024lt}.
It remains to show \ASN{thm:shape:c}.
It follows from \Cref{lem:gm}
that there exists an induced great morphism~$\xi\c S^1\times S^1\to Z$.
We know from \Cref{lem:ind}\ref{lem:ind:a} that $Z$ satisfies the Conditions
\ref{C1}, \ref{C2}, \ref{C3} and \ref{C4}.
By \Cref{lem:c123} either \ref{c1}, \ref{c2} or \ref{c3} holds.
We deduce from \Cref{lem:gm} and \Cref{lem:ind}\ref{lem:ind:b} that these
correspond to Shape~\ref{I}, Shape~\ref{II} and Shape~\ref{III}, \resp
~(see also \Cref{exm:shapeI}).
\end{proof}

\begin{proof}[Proof of \Cref{cor:top}]
If $\deg Z=2$, then $Z$ is a sphere.

If $Z$ is either a EO cyclide or CO cyclide, then it stereographically
projects to a quadratic cone and thus there exists spheres $S$ and $S'$
\st $Z=S\cup S'$ and $|S\cap S'|=2$.

Next, suppose that $Z$ is either a Blum cyclide, Perseus cyclide or ring cyclide.
Let $X\subset\S^3$ be the great celestial surface \st $X(\R)=Z$.
By \citep[Theorem~3]{2021circle}, there exists
a smooth model~$\varphi\c Y\to X$ \st $Y$ is $\P^1\times\P^1$ blown up
in non-real points.
This implies that $Y_\R\cong\P^1_\R\times\P^1_\R\cong S^1\times S^1$.
Since $X(\R)$ is smooth, the map~$\varphi$ restricted to the real points
defines a biregular isomorphism~$Y_\R\to X(\R)$.
Hence, $Z$ is in this case a torus.

Finally, if $Z$ has Shape~\ref{I}, Shape~\ref{II} or Shape~\ref{III},
then $Z$ is as a direct consequence of the definitions as asserted
in \Cref{cor:top}. In particular, recall that by \citep[Theorem~C13]{2003},
any two non-trivial cycles in a torus~$T$ are related by some homeomorphism~$T\to T$.
By \Cref{thm:shape}, we considered all possible cases and thus concluded the proof.
\end{proof}

\begin{proof}[Proof of \Cref{cor:prod}]
Direct consequence of \Cref{thm:shape}. Notice that
an EO cyclide and CO~cyclide each contain infinitely many concurrent circles.
\end{proof}

\begin{proof}[Proof of \Cref{cor:AB}]
Direct consequence of \Cref{thm:shape}.
Notice that $\{a\}\star B$ and $A\star\{b\}$ are circles for all $a\in A$, $b\in B$ and thus
$A\star B$ is a great celestial surface.
\end{proof}

The following corollary proposes
an alternative for the hypothesis that~$X\subset\S^3$
is a great celestial surface.

\begin{corollary}
\label{cor:map}
If there exists a birational morphism $\varphi\c \P^1\times\P^1\to X\subset \S^3$ \st
\begin{Mlist}
\item
$\varphi(\{p\}\times\P^1)$ and $\varphi(\P^1\times\{p\})$ are circles for all points $p\in\P^1_\R$, and
\item
$\varphi(\{q\}\times\P^1)\cap X_\R=\Sing X_\R$ is a double circle for some complex point~$q\in\P^1$,
\end{Mlist}
then there exist tori~$T,T'\subset S^3$ and circle $C\subset S^3$ \st
$C=\Sing X(\R)$ and either
\begin{Mlist}
\item $X(\R)=T\cup T'$ with $C=T\cap T$ a non-trivial cycle in both $T$ and $T'$,
\item $X(\R)=T$ with $C\subset T$, or
\item $X(\R)=T\cup C$ and $T\cap C=\varnothing$.
\end{Mlist}

\end{corollary}

\begin{proof}
Since $\Sing X_\R$ is a double circle, we deduce that $\varphi$ is a great morphism.
Hence, the proof is concluded by \Cref{lem:c123}.
\end{proof}

\begin{conjecture}
If $\varphi\c \P^1\times\P^1\to X\subset \S^3$ is a great morphism,
then $X$ is M\"obius equivalent to a great surface.
\end{conjecture}

\section{Acknowledgements}

I thank Mikhail~Skopenkov for the many interesting discussions, insights and corrections.
Thanks also go to Helmut~Pottmann for the insights and inspiration
concerning the classical geometric aspects of celestial surfaces and their applications in architecture.
Finally, many thanks to the editor and anonymous referees for detailed comments.
The surface figures were generated using \citep[Sage]{2012sage}, \citep[Surfex]{2008surfex} and Mathematica.
This research was supported
by the Austrian Science Fund (FWF) projects P33003 and P36689.

\bibliography{great-circle}

\ifarxiv
\textbf{address:}
Institute for Algebra, Johannes Kepler University, Linz, Austria
\\
\textbf{email:} info@nielslubbes.com
\fi

\end{document}